\documentclass[11pt,reqno,twoside]{amsart}
\usepackage{graphicx}
\usepackage{amssymb}
\usepackage{epstopdf}
\usepackage[asymmetric,top=3.5cm,bottom=4.3cm,left=3.1cm,right=3.1cm]{geometry}
\usepackage{graphicx}
\usepackage{bm}

\usepackage{booktabs} 
\usepackage{array} 
\usepackage{paralist} 
\usepackage{verbatim} 
\usepackage{subfig} 
\usepackage{tabularx}
\usepackage{amsmath,amsfonts,amsthm,mathrsfs,amssymb,cite}
\usepackage{ulem}
\usepackage{color}
\newtheorem{thm}{Theorem}[section]
\newtheorem{cor}[thm]{Corollary}
\newtheorem{lem}{Lemma}[section]
\newtheorem{prop}{Proposition}[section]
\theoremstyle{definition}
\newtheorem{defn}{Definition}[section]
\theoremstyle{remark}

\newtheorem{rem}{Remark}[section]
\numberwithin{equation}{section}

\newcommand{\Ccal}{\mathcal{C}}
\newcommand{\cf}{\mathcal{F}}

\newcommand{\ci}{\mathcal{I}}

\newcommand{\rmd}{\mathrm {d}}

\newcommand{\cl}{\mathcal{L}}
\newcommand{\ck}{\mathcal{K}}
\newcommand{\ct}{\mathcal{T}}
\newcommand{\bu}{\mathbf{u}}

\newcommand{\bv}{\mathbf{v}}
\newcommand{\bw}{\mathbf{w}}
\newcommand{\bff}{\mathbf{f}}

\newcommand{\bx}{\mathbf{x}}
\newcommand{\by}{\mathbf{y}}

\newcommand{\bH}{\mathbf{H}}

\newcommand{\bL}{\mathbf{L}}

\newcommand{\bg}{\mathbf{g}}
\newcommand{\bh}{\mathbf{h}}
\newcommand{\bp}{\mathbf{p}}
\newcommand{\bq}{\mathbf{q}}

\newcommand{\bpsi}{\bm{\psi}}
\newcommand{\bvarphi}{\bm{\varphi}}
\newcommand{\bphi}{\bm{\phi}}

\newcommand{\bxi}{\bm{\xi}}
\newcommand{\bnu}{\bm{\nu}}

\newcommand{\beq}{\begin{equation}}
\newcommand{\eeq}{\end{equation}}

\newcommand{\supp}{{\rm supp}}

\allowdisplaybreaks

\usepackage[capitalize,noabbrev]{cleveref}

\title[An effective medium theory for impenetrable elastic obstacles]
{Effective medium theory for embedded obstacles in elasticity with applications to inverse problems}

\author{Zhengjian Bai}
\address{School of Mathematical Sciences, Xiamen University, Xiamen 361005, P.R. China}
\email{zjbai@xmu.edu.cn}

\author{Huaian Diao}
\address{School of Mathematics and Statistics, Northeast Normal University,  Changchun 130024\\
 \indent\,\,P.R. China}
\email{hadiao@nenu.edu.cn}

\author{Hongyu Liu}
\address{Department of Mathematics, City University of Hong Kong, Kowloon, Hong Kong, China}
\email{hongyu.liuip@gmail.com; hongyuliu@hkbu.edu.hk}

\author{Qingle Meng}
\address{School of Mathematical Sciences, Xiamen University, Xiamen 361005, P.R. China}
\email{qinglemeng@yahoo.com}

\begin{document}
\maketitle

\begin{abstract}

We consider the time-harmonic elastic wave scattering from a general (possibly anisotropic) inhomogeneous medium with an embedded impenetrable obstacle. We show that the impenetrable obstacle can be effectively approximated by an isotropic elastic medium with a particular choice of material parameters. We derive sharp estimates to rigorously verify such an effective approximation. Our study is strongly motivated by the related studies of two challenging inverse elastic problems including the inverse boundary problem with partial data and the inverse scattering problem of recovering mediums with buried obstacles. The proposed effective medium theory readily yields some interesting applications of practical significance to these inverse problems.

\medskip

\medskip



\end{abstract}

\section{Introduction}

\subsection{Motivations and background}\label{subsect:motivation}
Our study is strongly motivated by the related studies of two challenging inverse elastic problems, which we shall discuss in what follows. To that end, we first introduce the Lam\'e system that governs the elastic wave propagation in $\mathbb{R}^n$, $n=2, 3$. Throughout, we let $\mathcal{C}$ and $\rho$ signify the constitutive material parameters of an elastic medium. Here, $\mathcal{C}(\mathbf{x})=(\mathcal{C}_{ijkl}(\mathbf{x}))_{i,j,k,l=1}^n$ is a four-rank real-valued tensor satisfying the following symmetry property:
\begin{equation}\label{eq:symm1}
\mathcal{C}_{ijkl}=\mathcal{C}_{klij}\quad \mbox{and}\quad \mathcal{C}_{ijkl}=\mathcal{C}_{jikl}=\mathcal{C}_{ijlk},\quad i,j,k,l=1, 2, \ldots, n.
\end{equation}
$\rho(\mathbf{x})$ is a bounded measurable complex-valued function with $\Re \rho>0$ and $\Im \rho\geq 0$. Physically, $\mathcal{C}$ signifies the stiffness tensor, and $\Re\rho$ and $\Im\rho$ characterize the density and damping of an elastic medium, respectively. Let $\mathbf{u}(\mathbf{x})=(u_j(\mathbf{x}))_{j=1}^n\in\mathbb{C}^n$ denote the displacement field in the elastic medium. In linear elasticity, one has the following Lam\'e system:
\begin{equation}\label{eq:lame1}
\mathcal{L}_{\mathcal{C}}\mathbf{u}+\omega^2\rho \mathbf{u}={\bf0}, \quad \mathcal{L}_{\mathcal{C}}\mathbf{u}:=\nabla\cdot(\mathcal{C}:\nabla\mathbf{u})=\left(\sum_{j,k,l=1}^n\partial_j(\mathcal{C}_{ijkl}\partial_l u_k)\right)_{i=1}^n,
\end{equation}
where $\omega\in\mathbb{R}_+$ signifies the angular frequency and $\mathcal{L}_{\mathcal{C}}$ is referred to as the Lam\'e operator associated with $\mathcal{C}$. In \eqref{eq:lame1}, the symbol $``:"$ indicates an action of double contraction, which is defined for two matrices $\mathbf{A}=(a_{ij})_{i,j=1}^n$ and $\mathbf{B}=(b_{ij})_{i,j=1}^n$:
\[
\mathbf{A}:\mathbf{B}=\sum_{i,j=1}^n a_{ij} b_{ij}\quad\mbox{and}\quad \mathcal{C}:\mathbf{A}=(\mathcal{C}:\mathbf{A})_{ij}=\left(\sum_{k,l=1}^n\mathcal{C}_{ijkl}a_{kl}\right).
\]
Throughout we assume that the elastic tensor $\mathcal{C}$ satisfies the uniform Legendre ellipticity condition:
\begin{equation}\label{eq:ellip1}
c_{\text{min}}\|\bxi\|_2^2 \leq \bxi:\mathcal{C}:\bxi^*\leq c_{\text{max}}\|\bxi\|_2^2,\quad \forall \,\bxi\in\mathbb{C}^{n\times n}\ \ \mbox{being a symmetric matrix},
\end{equation}
where $c_{\text{min}}$ and $c_{\text{max}}$ are two positive constants. If there exist scalar real functions $\lambda(\mathbf{x})$ and $\mu(\mathbf{x})$ such that
\begin{equation}\label{eq:lame2}
\mathcal{C}_{ijkl}=\lambda\delta_{ij}\delta_{kl}+\mu(\delta_{ik}\delta_{jl}+\delta_{il}\delta_{jk}),
\end{equation}
where $\delta$ is the Kronecker delta function, then the elastic medium is said to be isotropic, otherwise it is said anisotropic.

Let $\Sigma\Subset\mathbb{R}^n$ be a bounded Lipschitz domain. Consider the following boundary value problem associated with the Lam\'e system:
\begin{equation}\label{eq:lame3}
\mathbf{u}\in H^1(\Sigma)^n,\quad \mathcal{L}_{\mathcal{C}}\mathbf{u}+\omega^2\rho\mathbf{u}={\bf0}\ \ \mbox{in}\ \ \Sigma,\quad \mathcal{T}_{\bnu}(\mathbf{u})=\bpsi\in H^{-1/2}(\partial\Sigma)^n\ \ \mbox{on}\ \ \partial\Sigma,
\end{equation}
where $\mathcal{T}_{\bnu}(\mathbf{u}):=\bnu\cdot(\mathcal{C}:\nabla\mathbf{u})$ with $\bnu\in\mathbb{S}^{n-1}$ signifying the exterior unit normal vector to $\partial\Sigma$. It is known that there exists a unique solution to \eqref{eq:lame3}, provided that $\omega$ does not belong to a discrete set (known as the eigenvalues) \cite{Mclean2000}. Assuming that $\omega$ is not an eigenvalue, the following boundary Neumann-to-Dirichlet (NtD) map is well-defined:
\begin{equation}\label{eq:NtD1}
\Lambda_{\Sigma; \mathcal{C}, \rho}: H^{-1/2}(\partial\Sigma)^n\mapsto H^{1/2}(\partial\Sigma)^n, \quad \Lambda_{\Sigma; \mathcal{C}, \rho}(\bpsi)=\mathbf{u}|_{\partial\Sigma},
\end{equation}
where $\mathbf{u}$ is the solution to \eqref{eq:lame3}. The NtD map $\Lambda_{\mathcal{C}, \rho}$ encodes all the possible Cauchy data $(\mathcal{T}_{\bnu}(\mathbf{u})|_{\partial\Sigma}, \mathbf{u}|_{\partial\Sigma})$ associated with the Lam\'e system \eqref{eq:lame3}. An inverse problem of industrial importance arising in the elastic probing is to recover the elastic body $(\Sigma; \mathcal{C}, \rho)$ by the boundary observations, namely:
\begin{equation}\label{eq:ip1}
\Lambda_{\Sigma; \mathcal{C}, \rho}\rightarrow (\Sigma; \mathcal{C}, \rho).
\end{equation}
In practice, it means that one exerts the traction force on the boundary of the elastic body (i.e. $\mathcal{T}_{\bnu}(\mathbf{u})|_{\partial\Sigma}=\bpsi$) to induce the elastic field $\mathbf{u}$ inside the body, and then measures the response on the boundary (i.e. $\mathbf{u}$), and in such a non-destructive way to infer knowledge of the interior of the elastic body. The inverse problem \eqref{eq:ip1} is nonlinear and ill-conditioned and has been extensively and intensively investigated in the literature, see e.g. \cite{Haher1998, BFV14, DFS20, ER02, HL2014, HL2015, NG, NGW} and the references cited therein. In many practical scenarios, one cannot achieve the measurements of the elastic field on the full boundary $\partial \Omega$, and instead, one can only measure on  part of the boundary, say $(\mathcal{T}_{\bnu}(\mathbf{u})|_{\Gamma}, \mathbf{u}|_{\Gamma})$, where $\Gamma\Subset\partial\Sigma$. This is particular the case that $\Sigma$ is not a solid body and possesses a hole, say $\Sigma=\Omega\backslash\overline{D}$, where $D\Subset\Omega$, and $\Omega$ and $D$ are both solid bodies.\footnote{One can think that in two dimensions, both $\Omega$ and $D$ are simply connected. } In such a case, $\partial\Sigma=\partial\Sigma_{\text{interior}}\cup\partial\Sigma_{\text{exterior}}$, where the interior boundary $\partial\Sigma_{\text{interior}}=\partial D$ and the exterior boundary $\partial\Sigma_{\text{exterior}}=\partial\Omega$. From a practical point of view, the interior boundary is inaccessible in the elastic probing, and hence in the inverse problem \eqref{eq:ip1}, one can only exert the input and measure the output on the exterior boundary, namely $\Gamma=\partial\Omega$. That is, one needs to require in \eqref{eq:lame3} that $\mathrm{supp}(\psi)\subset\partial\Sigma_{\text{exterior}}=\partial\Omega$, which leads to the following system:
\begin{equation}\label{eq:lame4}
\mathcal{L}_{\mathcal{C}}\mathbf{u}+\omega^2\rho\mathbf{u}={\bf0}\ \ \mbox{in}\ \ \Omega\backslash\overline{D},\quad \mathcal{T}_{\bnu}(\mathbf{u})={\bf0} \ \ \mbox{on}\ \ \partial D, \quad \mathcal{T}_{\bnu} (\mathbf{u})=\bpsi\in H^{-1/2}(\partial\Omega)^n\ \ \mbox{on}\ \ \partial\Omega.
\end{equation}
Then the inverse problem \eqref{eq:ip1} becomes:
\begin{equation}\label{eq:ip1p}
(\mathcal{T}_{\bnu} (\mathbf{u})|_{\partial\Omega}, \mathbf{u}|_{\partial\Omega})\rightarrow (\Omega\backslash\overline{D}; \mathcal{C}, \rho),
\end{equation}
where $\mathbf{u}\in H^1(\Omega\backslash\overline{D})^n$ is the solution to \eqref{eq:lame4}. The partial-data inverse problem constitutes a class of highly challenging open problems in the literature, and it even remains largely open for the case associated with the differential equation $\nabla\cdot(\sigma\nabla u)=0$ where $\sigma$ is a scalar function \cite{IUY10,KSU} (the so-called Calder\'on's inverse conductivity problem), a fortiori the one associated with the Lam\'e system \eqref{eq:lame4}. We are aware that the partial-data inverse elastic problem was recently studied in \cite{DFS20} following the spirit of the related studies of the partial-data Calder\'on problem within a certain restricted and special setup.

In this paper, we propose a different perspective to tackle the partial-data inverse elastic problem that can work in an extremely general scenario. To that end, we note that physically, $D$ represents a traction-free impenetrable obstacle embedded in the elastic medium $(\Omega\backslash\overline{D}; \mathcal{C}, \rho)$. In what follows, we set $D\oplus(\Omega\backslash\overline{D}; \mathcal{C}, \rho)$ to signify such an elastic object as described above. Let $\Lambda^p_{\Omega\backslash\overline{D}; \mathcal{C}, \rho}: H^{-1/2}(\partial\Omega)^n\mapsto H^{1/2}(\partial\Omega)^n$ denote the partial NtD map associated with the Lam\'e system \eqref{eq:lame4}. That is,
\begin{equation}\label{eq:pntd}
\Lambda^{p}_{\Omega\backslash\overline{D}; \mathcal{C}, \rho}(\bpsi)=\mathbf{u}|_{\partial\Omega},
\end{equation}
where $\mathbf{u}\in H^1(\Omega\backslash\overline{D})^n$ is the solution to \eqref{eq:lame4}. For comparison, we also write $\Lambda^f_{\Omega\backslash\overline{D}; \mathcal{C}, \rho}=\Lambda_{\Sigma; \mathcal{C}, \rho}$, where $\Lambda_{\Sigma; \mathcal{C}, \rho}$ is defined in \eqref{eq:NtD1}, to signify that it encodes the full boundary measurements.
\begin{defn}\label{def:1}
Consider $D\oplus(\Omega\backslash\overline{D}; \mathcal{C}, \rho)$ as described above. If there exist an elastic medium $(\Omega; \widetilde{\mathcal{C}}, \widetilde{\rho})$ with $(\widetilde{\mathcal{C}}, \widetilde{\rho})\big|_{\Omega\backslash\overline{D}}=(\mathcal{C}, \rho)\big|_{\Omega\backslash\overline{D}}$, and $\varepsilon\in\mathbb{R}_+$ with $\varepsilon\ll 1$ such that
\begin{equation}\label{eq:effective1}
\left\|  \Lambda^{p}_{\Omega\backslash\overline{D}; \mathcal{C}, \rho}-\Lambda^f_{\Omega; \widetilde{\mathcal{C}}, \widetilde{\rho} }  \right\|_{\mathcal{L}\big(H^{-1/2}(\partial\Omega)^n, H^{1/2}(\partial\Omega)^n\big) }\leq C \varepsilon,
\end{equation}
where $C$ is a generic positive constant depending on $\Omega, D$ and $\mathcal{C}, \rho, \omega$, then $(\Omega; \widetilde{\mathcal{C}}, \widetilde{\rho})$ is said to be an effective $\varepsilon$-realization of $D\oplus(\Omega\backslash\overline{D}; \mathcal{C}, \rho)$. If $\varepsilon\equiv 0$, then $(\Omega; \widetilde{\mathcal{C}}, \widetilde{\rho})$ is said to be an effective realization of $D\oplus(\Omega\backslash\overline{D}; \mathcal{C}, \rho)$.

\end{defn}

It is conjectured that the unique identifiability holds generically for the aforementioned partial-data inverse problem, namely the correspondence between $ \Lambda^{p}_{\Omega\backslash\overline{D}; \mathcal{C}, \rho}$ and $D\oplus(\Omega\backslash\overline{D}; \mathcal{C}, \rho)$ is one-to-one. It means that the (perfect) effective realization of $D\oplus(\Omega\backslash\overline{D}; \mathcal{C}, \rho)$ should not exist in generic scenarios. However, we shall show in this paper that there always exist effective $\varepsilon$-realizations of $D\oplus(\Omega\backslash\overline{D}; \mathcal{C}, \rho)$ for any given $\varepsilon\ll 1$. If so, the partial-data inverse problem of recovering $(\Omega\backslash\overline{D}; \mathcal{C}, \rho)$ by knowledge of $ \Lambda^{p}_{\Omega\backslash\overline{D}; \mathcal{C}, \rho}$ can be (at least approximately) reduced to
the full-data inverse problem of recovering $(\Omega; \widetilde{\mathcal{C}}, \widetilde{\rho})$ by knowledge of $\Lambda^f_{\Omega; \widetilde{\mathcal{C}}, \widetilde{\rho} }$, whereby one has rich results for the unique identifiability and reconstruction methods; see the references cited earlier as well as the references therein. We shall present more discussions in what follows on the interesting implications of our study to the inverse elastic problem.

So far, we have mainly considered the inverse boundary problem of making use of the traction field $\mathcal{T}_{\bnu}(\mathbf{u})$ as the boundary input and the displacement field $\mathbf{u}$ on the boundary as the measured output. An alternative way is to make use of the displacement field as the boundary input and the boundary traction field as the output. By following a similar discussion, one can show that the homogeneous condition $\mathcal{T}_{\bnu}(\mathbf{u})|_{\partial D}={\bf0}$ should be replaced by $\mathbf{u}|_{D}={\bf0}$. In such a case, $D$ is referred to as a rigid impenetrable obstacle in the literature. Clearly, Definition~\ref{def:1} also applies for the (perfect or approximate) effective realization of an embedded rigid obstacle.

Another inverse problem of close interest is the simultaneous recovery of buried obstacles and surrounding mediums in the elastic scattering theory. Let $D\oplus(\Omega\backslash\overline{D}; \mathcal{C}, \rho)$ be described as earlier, where $D$ can either a traction-free or a rigid obstacle. Let $\lambda_e, \mu_e$ and $\rho_e$ be real constants satisfying the strong convexity condition (induced by the ellipticity condition \eqref{eq:ellip1}):
\begin{equation}\label{eq:convexity1}
\mu_e>0, \quad n\lambda_e+2\mu_e>0\quad \mbox{and}\quad \rho_e>0.
\end{equation}
Let $\mathcal{C}^e$ be an isotropic elastic tensor as defined in \eqref{eq:lame2} with $\lambda=\lambda_e$ and $\mu=\mu_e$. Let $(\mathcal{C}, \rho)$ be extended into $\mathbb{R}^n\backslash\overline{\Omega}$ such that $(\mathcal{C}, \rho)=(\mathcal{C}^e, \rho_e)$ in $\mathbb{R}^n\backslash\overline{\Omega}$. Let $\mathbf{u}^{in}$ be an entire solution to the following Lam\'e system:
\begin{equation}\label{eq:incident1}
\mu_e\Delta\mathbf{u}^{in}+(\lambda_e+\mu_e)\nabla(\nabla\cdot\mathbf{u}^{in})+\omega^2\rho_e\mathbf{u}^{in}={\bf0}.
\end{equation}
Consider the following elastic scattering system:
\begin{equation}\label{eq:scattering1}
\begin{cases}
\mathcal{L}_{\mathcal{C}}\mathbf{u}+\omega^2\rho\mathbf{u}=\bff& \mbox{in}\ \ \mathbb{R}^n\backslash\overline{D},\medskip\\
\mathbf{u}=\mathbf{u}^{in}+\mathbf{u}^s & \mbox{in}\ \ \mathbb{R}^n\backslash\overline{\Omega},\medskip\\
\mathcal{B}(\mathbf{u})={\bf0} & \mbox{on}\ \ \partial D,\medskip\\
\bu\big|_{\partial \Omega}=\bu^{s}\big|_{\partial \Omega}+\bu^{in},\quad \ct_{\bnu}(\bu)=\ct_{\bnu}(\bu^{s})+\ct_{\bnu}(\bu^{in})&\mbox{on}\ \  \partial{\Omega},\\
\bu^{\mathrm p,\,s}=-\frac{1}{k_{\mathrm p}^2} \nabla (\nabla \cdot \bu^{s} ),\quad \bu^{\mathrm s,\,s}= \frac{1}{k_{\mathrm s} ^2}  \nabla \times  ( \nabla \times   \bu^{s} ) & \mbox{in}\ \ \mathbb{R}^n\backslash\overline{\Omega}, \medskip\\
{\lim_{|\bx| \to\infty}|\bx|^{(n-1)/2}\big(  \frac{\partial{\bu^{\mathrm { t},s}}}{\partial{|\bx|}}   -\imath\kappa_{\mathrm  t}\bu^{\mathrm  t,s}  \big)=\bf 0,} & {\mathrm  t=\mathrm  p,\mathrm s,} \
\end{cases}
\end{equation}
where  $\bff(\bx)$ indicates a source and is compactly supported outside $\Omega$, namely $\supp(\bff)\subset B_{r_0}\backslash \overline{\Omega}$ for some ball $B_{r_0}$ with center at the origin and a radius of $r_0$. $\imath:=\sqrt{-1}$, $\kappa_{\mathrm s}:=\omega\sqrt{1/\mu_e}$ and $\kappa_{\mathrm p}:=\omega\sqrt{1/(\lambda_e+2\mu_e)}$, and $\mathcal{B}(\mathbf{u})=\mathbf{u}$ or $\mathcal{B}(\mathbf{u})=\mathcal{T}_{\bnu}(\mathbf{u})$ correspond, respectively, to the cases that $D$ is rigid or traction-free. The system \eqref{eq:scattering1} describes the time-harmonic scattering due to an incident field $\mathbf{u}^{in}$ and the scatter $D\oplus(\Omega\backslash\overline{D}; \mathcal{C}, \rho)$. $\mathbf{u}^s$ is referred to as the scattered field, which characterizes the perturbation of the propagation of the incident field due to the presence of the inhomogeneous scatterer. $\mathbf{u}^{\mathrm{p}, s}$ and $\mathbf{u}^{\mathrm{s}, s}$ are the compressional and shear parts of $\mathbf{u}^s$, respectively. The last limit in \eqref{eq:scattering1} is known as the Kupradze radiation condition, which holds uniformly in the angular variable $\hat{\mathbf{x}}:=\mathbf{x}/|\mathbf{x}|\in\mathbb{S}^{n-1}$. We couldn't find a convenient reference for the well-posedness of the scattering problem \eqref{eq:scattering1} in such a general scenario and shall provide a proof in Subsection \ref{well-posedness}. The solution $\mathbf{u}^s\in H_{loc}^1(\mathbb{R}^n\backslash\overline{D})^n$ admits the following asymptotic expansion (cf. \cite{Haher1998}):
\begin{align}\label{ineq:farfiel}
\bu^s(\bx)=\dfrac{\exp(\imath\kappa_{\mathrm p} |\bx|)}{|\bx|^{(n-1)/2}}\bu^{\mathrm  p,\infty}(\hat {\bx})+\dfrac{\exp(\imath\kappa_{\mathrm s}|\bx|)}{|\bx|^{(n-1)/2}}\bu^{\mathrm s,\infty}(\hat {\bx})+\mathcal{O}(|\bx|^{-\frac{n+1}{2}})\quad \mbox{as}\quad |\bx|\to\infty,
\end{align}
uniformly in all directions $\hat{\bx}\in \mathbb{S}^{n-1}$. $\bu^{\mathrm p,\infty}$ and $\bu^{\mathrm s,\infty}$ are defined on the unit sphere $\mathbb{S}^{n-1}$, and are known as the longitudinal and transversal far field patterns corresponding to $\bu^{\mathrm p, s}$ and $\bu^{\mathrm s,s}$, respectively. The far-field pattern $\bu^\infty$ of the scattered field $\bu^s$ is defined as the sum of $\bu^{\mathrm p,\infty}$ and $\bu^{\mathrm s,\infty}$, i.e.,
$$
\bu^\infty:=\bu^{\mathrm p,\infty}+\bu^{\mathrm s,\infty}.
$$
It is known that $\bu^{\mathrm p,\infty}$ is normal to $\mathbb{S}^{n-1}$ and $\bu^{\mathrm s,\infty}$
is tangential to $\mathbb{S}^{n-1}$. Thus we have $\bu^{\mathrm p,\infty}=(\bu^{\infty}\cdot\hat{\bx})\,\hat{\bx}$ and $\bu^{\mathrm s,\infty}=\hat{\bx}\times \bu^{\infty}\times\hat{\bx}$.

An inverse scattering problem arsing in practical applications including seismology and elastography is to recover the inhomogeneous scatterer $D\oplus(\Omega\backslash\overline{D}; \mathcal{C}, \rho)$ by knowledge of the associated far-field pattern $\mathbf{u}^\infty$, namely,
\begin{equation}\label{eq:ip2}
\mathbf{u}^\infty(\hat{\mathbf{x}}; \mathbf{u}^{in}, \mathbf{f}, D\oplus(\Omega\backslash\overline{D}; \mathcal{C}, \rho) )\rightarrow D\oplus(\Omega\backslash\overline{D}; \mathcal{C}, \rho).
\end{equation}
Here, it is noted that there are two different kinds of sources, $f$ and $\mathbf{u}^{in}$, in \eqref{eq:incident1}, which correspond to the so-called passive and active measurements in the context of the inverse problem \eqref{eq:ip2}. In order to give a general study, we include both of them into our study, and either one of them can be taken to be zero, corresponding to different scenarios in the context of the inverse problem \eqref{eq:ip2} in the literature. In \eqref{eq:ip2}, the presence of the impenetrable obstacle $D$ make the study of the inverse problem radically more challenging compared to the case without the obstacle, i.e. $D=\emptyset$. In fact, to our best knowledge, there is no result available in the literature for the inverse problem \eqref{eq:ip2} in the case when $D\neq \emptyset$, whereas there are rich results in the case $D=\emptyset$; see e.g. \cite{Haher1998,Hahenr1993} and the references cited therein. Nevertheless, we would like to mention some related studies for the inverse acoustic and electromagnetic scattering problems in simultaneously recovering a buried obstacle and its surrounding medium \cite{DLL,LL17,Liu2012} where one needs to make use of multiple-frequency measurements, namely severely over-determined data were used. Similar to the treatment for the inverse boundary problem \eqref{eq:ip1} with partial measurements, we intend reduce the inverse scattering problem \eqref{eq:ip2} to a simpler case with no buried obstacles in an effective way. To that end, we introduce the following definition.
\begin{defn}\label{def:2}
Consider $D\oplus(\Omega\backslash\overline{D}; \mathcal{C}, \rho)$ as described above. If there exist an elastic medium $(\Omega; \widetilde{\mathcal{C}}, \widetilde{\rho})$ with $(\widetilde{\mathcal{C}}, \widetilde{\rho})\big|_{\Omega\backslash\overline{D}}=(\mathcal{C}, \rho)\big|_{\Omega\backslash\overline{D}}$, and $\varepsilon\in\mathbb{R}_+$ with $\varepsilon\ll 1$ such that
\begin{equation}\label{eq:effective1}
\begin{split}
&\big\|\mathbf{u}^\infty(\hat{\mathbf{x}}; \mathbf{u}^{in}, \mathbf{f}, (\Omega; \widetilde{\mathcal{C}}, \widetilde{\rho}) )-\mathbf{u}^\infty(\hat{\mathbf{x}}; \mathbf{u}^{in}, \mathbf{f}, D\oplus(\Omega\backslash\overline{D}; \mathcal{C}, \rho) ) \big\|_{C(\mathbb{S}^{n-1})^n} \\
 &\leq C \varepsilon\left(\|\mathbf{u}^{in}\|_{H^1(B_{r_0})^n}+\|\bff\|_{L^2(B_{r_0})^n}\right),
\end{split}
\end{equation}
where $B_{r_0}$ is any given central ball containing $\Omega$, and $C$ is a generic positive constant depending on the a-priori paramters, then $(\Omega; \widetilde{\mathcal{C}}, \widetilde{\rho})$ is said to be an effective $\varepsilon$-realization of $D\oplus(\Omega\backslash\overline{D}; \mathcal{C}, \rho)$. If $\varepsilon\equiv 0$, then $(\Omega; \widetilde{\mathcal{C}}, \widetilde{\rho})$ is said to be an effective realization of $D\oplus(\Omega\backslash\overline{D}; \mathcal{C}, \rho)$. For simpler terminologies, we also call $(D; \widetilde{\mathcal{C}}, \widetilde{\rho})$ an effective realization of the obstacle $D$.

\end{defn}

Hence, if one can find an effective realization of the embedded obstacle $D$, the inverse problem \eqref{eq:ip2} can then be effectively reduced to the recovery of $(\Omega; \widetilde{\mathcal{C}}, \widetilde{\rho})$, which possesses a much simpler topological structure.

\subsection{Summary of the main results}\label{subsect:summary of results}

Motivated by the studies of the inverse problems discussed earlier, we establish in this paper that there are always approximate effective realizations of the embedded obstacles. It is clear that the two problems \eqref{eq:lame4} and \eqref{eq:scattering1} are closely related. Indeed, they are equivalent if an appropriate truncation is introduced for truncating the unbounded domain $\mathbb{R}^n\backslash\overline{D}$ in \eqref{eq:scattering1} into a bounded one. In the rest of our paper, we shall present our study mainly for the scattering system \eqref{eq:scattering1}. On the one hand, the scattering model \eqref{eq:scattering1} is physically more relevant in the context of the inverse elastic problem study, and on the other hand, the corresponding mathematical argument for the effective medium theory associated with \eqref{eq:scattering1} is technically more involved that associated with \eqref{eq:lame4}. Nevertheless, it is emphasized that the results established in our
study hold equally for the corresponding problem associated with \eqref{eq:lame4}.

Our main result can be summarized in the following theorem.

\begin{thm}\label{thm:main1}
Consider the scattering problem \eqref{eq:scattering1} associated with the scatterer  $D\oplus(\Omega\backslash\overline{D}; \mathcal{C}, \rho)$. Let $\lambda_0, \mu_0$ be real constants satisfying the strong convexity condition in \eqref{eq:convexity1}, and $\eta_0, \tau_0$ be positive constants. Let $\varepsilon\in\mathbb{R}_+$ and $\varepsilon\ll 1$.
\begin{itemize}
	\item[{\rm (I)}] Case 1: If $D$ is a traction-free obstacle, then $(D; \mathcal{C}^0, \rho_0)$ with $\mathcal{C}^0$ given in the form \eqref{eq:lame2}:
\begin{equation}\label{eq:eff2}
\lambda=\varepsilon\lambda_0, \quad \mu=\varepsilon\mu_0,\quad\rho_0=\eta_0+\imath\tau_0,
\end{equation}
is an $\varepsilon^{1/2}$-realization of $D$ in the sense of Definition~\ref{def:2};
\item[{\rm (II)}]  Case 2: If $D$ is a rigid obstacle, then $(D; \mathcal{C}^0, \rho_0)$ with $\mathcal{C}^0$ given in the form \eqref{eq:lame2}:
\begin{equation}\label{eq:eff2_2}
\lambda=\varepsilon^{-2}\lambda_0, \,\, \mu=\varepsilon^{-2}\mu_0,\,\,\rho_0=(\eta_0+\imath\varepsilon^{-1}\tau_0),
\end{equation}
is an $\varepsilon^{1/2}$-realization of $D$ in the sense of Definition~\ref{def:2}.
\end{itemize}

\end{thm}

\begin{rem}\label{rem:1}
In \eqref{eq:eff2} and \eqref{eq:eff2_2}, we assume that $\lambda_0, \mu_0$ and $\eta_0, \tau_0$ are all constants. Indeed, they can be replaced to be variable functions satisfying the strong convexity condition and this can easily seen from our subsequent argument in proving Theorem~\ref{thm:main1}. However, we stick to the simpler case with constants in order to ease the exposition. The main point is that if $D$ is a traction-free obstacle, as long as the effective medium is lossy with asymptotically small bulk moduli, one can have the approximate effective realization effect. A similar remark can be made for the case if $D$ is a rigid obstacle.
\end{rem}

\subsection{Discussion}

We present more discussion on the implications of Theorem~\ref{thm:main1} to the inverse problem \eqref{eq:ip1p} or \eqref{eq:ip2}. As remarked earlier, we focus our discussion on \eqref{eq:ip2}. A standard approach for solving the inverse problem \eqref{eq:ip2} is the following optimization formulation:
\begin{equation}\label{eq:opt1}
\min_{\hat{D}\oplus(\Omega\backslash\overline{\hat{D}}; \hat{\mathcal{C}}, \hat{\rho})\in\mathscr{C}}\big\| \mathbf{u}^\infty(\hat{\mathbf{x}}; \mathbf{u}^{in}, \mathbf{f}, \hat{D}\oplus(\Omega\backslash\overline{\hat{D}}; \hat{\mathcal{C}}, \hat{\rho}) )-\mathcal{M}(D\oplus(\Omega\backslash\overline{D}; \mathcal{C}, \rho)) \big\|_{C(\mathbb{S}^{n-1})^n},
\end{equation}
where $\mathscr{C}$ and $\mathcal{M}$ signify the a-priori class of admissible scatterers and the measured far-field data, respectively. Clearly, $D\oplus(\Omega\backslash\overline{D}; \mathcal{C}, \rho)$ is a global minimizer to \eqref{eq:opt1}. By Theorem~\ref{thm:main1}, we replace \eqref{eq:opt1} by the following optimization problem for the reconstruction:
\begin{equation}\label{eq:opt2}
\min_{(\Omega; \hat{\mathcal{C}}, \hat{\rho})\in\mathscr{C}}\big\| \mathbf{u}^\infty(\hat{\mathbf{x}}; \mathbf{u}^{in}, \mathbf{f}, (\Omega; \hat{\mathcal{C}}, \hat{\rho}) )-\mathcal{M}(D\oplus(\Omega\backslash\overline{D}; \mathcal{C}, \rho)) \big\|_{C(\mathbb{S}^{n-1})^n}.
\end{equation}
By Definition~\ref{def:2} and Theorem~\ref{thm:main1}, $(\Omega; \widetilde{C}, \widetilde{\rho})$ is an asymptotically global minimizer to \eqref{eq:opt2}. Three remarks are in order. First, it is expected that in generic scenarios the reconstruction result from \eqref{eq:opt2} can (approximately) locate the topological defect of the underlying scatterer, namely the buried obstacle. In fact,  considering the case that $D$ is a traction-free obstacle, one can see that for the reconstructed medium, it should possess asymptotically small bulk moduli in the region where the obstacle is located. Second, for illustration, we only considered a simpler case with a single ``hole" above. It is clear that the same idea works for the case that there are multiple ``holes" within the scatterer. That is, one can start the reconstruction with the optimization formulation \eqref{eq:opt2} without any requirement of the a-priori knowledge of the topological structure of the underlying scatterer. Using the reconstruction result, one should be able to (approximately) profile the topological structure of the scatterer, namely to (approximately) identify the buried obstacles, by locating the regions where the reconstructed medium show a certain asymptotically peculiar behaviour. One can then use such a reconstruction result as an initial guess for the optimization formulation \eqref{eq:opt1} to further refine the reconstruction. Third, it is clear that the above described reconstruction procedure is rather heuristic. One would need to establish the uniqueness and stability results in order to guarantee the qualitative and quantitative properties of the minimizers to \eqref{eq:opt1} and \eqref{eq:opt2} required in the reconstruction procedure described above. The ill-posedness of the inverse problem shall add extra complexities to the desired theoretical justification. Hence, in this paper, in order to have a focusing theme of our study, we mainly consider the effective realization of embedded obstacles and postpone the more comprehensive inverse problem study in a forthcoming paper.

The rest of the paper is organized as follows. In Section \ref{sect:preliminary}, we mainly recall some preliminary results and give one important auxiliary lemma and give the proof of the well-poseness of \eqref{eq:scattering1}. The proof of Theorem \ref{thm:main1} for Case 1 and Case 2 will be provided in Section \ref{traction-free obstacle} and \ref{rigid obstacle}, respectively.

\section{Auxiliary results}\label{sect:preliminary}
\subsection{Preliminary}\label{subsect:preliminary}
In this subsection, we present some preliminary results for our subsequent use. We first recall the following lemma on the conormal derivative of the vector field  in the linear elasticity, which is a special case of Lemma 4.3 in \cite{Mclean2000}.

\begin{lem}\label{eq:cornormal} 
Let $\Omega\subset \mathbb{R}^n$ be a domain with a Lipschitz boundary.
Suppose  $\bu\in H^1(\Omega)^n$ and $\bh\in H^{-1}(\Omega)^n$ satisfying
\begin{align}\label{th:p0}
\mathcal{L}_{\mathcal{C}}\bu =\bh \quad\mbox{in}\quad \Omega,
\end{align}
where $\Ccal(\bx)$ is an elastic tensor satisfying the uniform Legendre ellipticity condition \eqref{eq:ellip1}.
Then there exists $\bg\in H^{-1/2}(\partial \Omega)^n$ such that
\begin{align}\label{def:innerproduct}
\Psi(\bu,\bv)=-(\bh,\bv)_{\Omega}+(\bg,\gamma\bv)_{\partial \Omega} \quad \forall \, \bv(\bx)\in H^{1}(\Omega)^n
\end{align}
with
\begin{align*}
\Psi(\bu,\bv)=\int_{\Omega}[\Ccal(\bx):\overline{\nabla\bu}]:\nabla\bv \rmd \bx,\,\,\,
(\bh,\bv)_{\Omega}=\int_{\Omega}\bh(\bx)\cdot \bv dx,\,\,\,
(\bg,\gamma\bv)_{\partial \Omega}=\int_{\partial \Omega}\bg\cdot\gamma\bv\rmd s(\bx),
\end{align*}
where $\gamma$ is the trace operator from $H^1(\Omega)^n$ to $H^{1/2}(\Omega)^n$.

Furthermore, $\bg$ is uniquely determined by $\bu$ and $\bh$ in the sense that  the following estimate holds for some constant $\eta>0$:
\begin{align}\label{ineq:estimate_conormal}
\|\bg\|_{H^{-1/2}(\partial \Omega)^n}\leq \eta\big(   \|\bu\|_{H^1(\Omega)^n}+\|\bh\|_{H^{-1}(\Omega)^n}        \big).
\end{align}
\end{lem}

In general, we write $\bg=\bnu\cdot[\Ccal(\bx):\nabla\bu]$ in the distribution sense, which is called the conormal derivative of $\bu$.
\begin{cor}\label{cor:conormal}
Let $\Omega\subset\mathbb{R}^n$ be a bounded connected domain with a Lipschitz boundary. Suppose $\bu$, $\bv$ $\in H^1(\Omega)^n$ and $\bh=-\omega^2\rho\bu\in L^2(\Omega)^n\Subset H^{-1}(\Omega)^n$. Then we have
\begin{align}\label{eq:estimate_conormal2}
\Psi(\bu,\bv)=(\omega^2\rho\bu,\bv)_{\Omega}+\big(\ct_{\bnu}(\bu),\,\gamma\bv\big)_{\partial \Omega}
\end{align}
and
\begin{align}\label{ineq:estimate_conormal2}
\|\ct_{\bnu}(\bu)\|_{H^{-1/2}(\partial \Omega)^n}\leq \,\eta\,  \|\bu\|_{H^1(\Omega)^n},
\end{align}
where $\eta$ is a positive constant.
\end{cor}
\begin{proof}
We can easily obtain  \eqref{eq:estimate_conormal2} by using \eqref{def:innerproduct}. Using the definitions of dual norms $\|\cdot\|_{H^{-1}(\Omega)^n}$ and $\|\cdot\|_{L^{2}(\Omega)^n}$ we have
\begin{align*}
\|\bh\|_{H^{-1}(\Omega)^n}&=\|\omega^2\rho\,\bu\|_{H^{-1}(\Omega)^n}=\sup_{{\bf0}\neq\mathbf w\in H^{1}(\Omega)^n}\dfrac{\big|(\omega^2\rho\bu,\,\mathbf w)_{\Omega}\big|}{\|\mathbf w\|_{H^{1}(\Omega)^n}}\leq\sup_{{\bf0}\neq\bw\in L^{2}(\Omega)^n}\dfrac{\big|(\omega^2\rho\bu,\,\mathbf w)_{\Omega}\big|}{\|\mathbf w\|_{L^{2}(\Omega)^n}}\\[2pt]
&=\|\omega^2\rho\,\bu\|_{L^{2}(\Omega)^n}\leq \eta\, \|\bu\|_{H^{1}(\Omega)^n},
\end{align*}
where  $\eta=\eta(\rho,\omega,\Omega)$ is a positive constant. Therefore, \eqref{ineq:estimate_conormal2} follows from \eqref{ineq:estimate_conormal}. The proof is complete.
\end{proof}

The next lemma, which can be named   as Rellich's lemma in the linear elasticity, can be proved by generalizing the arguments in \cite{Haher1998}.
\begin{lem}\label{Rellich}
Let $B_r$ be an appropriate ball  centered at origin with a  radius $r \in \mathbb R_+$, and assume that $\bu^s$ is a radiating solution to
$$
\mu \Delta\bu^s+(\lambda +\mu )\nabla( \nabla\cdot\bu^s)+ \omega^2\,\rho\bu^s={\bf 0},\quad \mu >0,\quad n \lambda +2 \mu >0,\quad\rho>0
$$
in $|\bx|\geq r$.
If
\begin{align}\label{ineq:rellich}
\Im\big( \int_{\partial B_r} \ct_{\bnu}\bu^s\cdot\overline{\bu^s} \rmd s  \big)\leq 0,
\end{align}
 then $\bu^s=\bf 0$ in $|\bx|\geq r$.
\end{lem}

By Lemma \ref{Rellich}, we can easily derive the following result.
\begin{thm}\label{th: thm1}
Suppose $\bu\in  H^1(\mathbb{R}^n)^n$ solves
\beq\label{du:r3}
\mathcal{L}_{\mathcal{C}}\bu+\omega^2\rho(\bx)\bu=\bf 0,
\eeq
where $\Ccal(\bx)$, $\omega$, and $\rho$ are defined as in \eqref{eq:lame1}. If $\bu$  satisfies the Kupradze radiation conditions, then $\bu$ vanishes in $\mathbb{R}^n$.
\end{thm}

\begin{proof}
Let $B_r$ be an appropriate large ball with center at origin and radius $r$. Multiplying $\overline{\bu}$ on both sides of \eqref{du:r3} and integrating over $B_r$, we have
\begin{align*}
\int_{B_r}\mathcal{L}_{\mathcal{C}}\bu\cdot \overline{\bu}\,\rmd\bx+\int_{B_r}\,\omega^2\rho(\bx)\bu\cdot \overline{\bu}\,\rmd\bx=0 .
\end{align*}
By applying Betti's first formula (cf. \cite{Alves1997}), we get
\begin{align*}
\int_{\partial B_r}\bnu\cdot(\Ccal(\bx):\nabla\bu)\cdot\overline{\bu}\,\rmd s(\bx)-\int_{B_r}(\Ccal(\bx):\nabla\overline{\bu}):\nabla\bu dx+\int_{B_r}\omega^2\rho(\bx)|\bu|^2\rmd\bx=0.
\end{align*}
Taking the imaginary part of the above equation, we obtain
$$
\Im\int_{\partial B_r}\bnu\cdot(\Ccal(\bx):\nabla \bu)\cdot\,\overline{\bu}\,\rmd s(\bx)=-\int_{B_r}\omega^2\,\Im \rho\,|\bv|^2 \,\rmd\bx\leq0.
$$
By Lemma \ref{Rellich}, we have $\bu=\bf 0$ outside $B_r$. Then it follows from  the unique continuation that $\bu=\bf 0$ in $\mathbb{R}^n$.
\end{proof}
\subsection{Auxiliary lemmas for Case 1}\label{subsect:auxiliary}
We derive several technical auxiliary lemmas for proving Theorem \ref{thm:main1} in Sections \ref{traction-free obstacle} and \ref{rigid obstacle}.
 We first consider Case 1 in Theorem~\ref{thm:main1}, where $D$ is a traction-free obstacle. In what follows, we let $B_r$ signify a central ball of radius $r$ containing $\Omega$, and consider the following two scattering problems:
Given $\bp\in H^{-1/2}(\partial D)^n$, $\bh_1\in H^{1/2}(\partial \Omega)^n$, $\bh_2\in H^{-1/2}(\partial \Omega)^n$ and $\bff $ with $\supp(\bff)\subset B_{r_0}\backslash\overline{\Omega}\subset B_{r}\backslash\overline{\Omega}$,
find $(\bv,\,\bu^s)\in H^1(\Omega\backslash\overline{D})^n\times H^1(\mathbb{R}^n\backslash\overline{\Omega})^n$ such that
\begin{align}\label{mode:transmission}
\left\{ \begin{array}{ll}
\mathcal{L}_{\mathcal{C}}\bv
+\omega^2\rho(\bx)\bv=\bf 0&\mbox{in}\ \  \Omega\backslash \overline{D},\quad \\[5pt]
\mathcal{L}_{\mathcal{C}^e}\bu^s
+\omega^2\rho_e\bu^s=\bff&\mbox{in}\ \ \mathbb{R}^n\backslash \overline{\Omega},\\[5pt]
\bu^s=\bu^{p,\,s}+\bu^{s,\,s}&\mbox{in}\ \ \mathbb{R}^n\backslash \overline{\Omega},\\[5pt]
\ct_{\bnu}(\bu)=\bp&\mbox{on}\ \ \partial{D},\\[5pt]
\bv=\bu^{s}+\bh_1,\quad \ct_{\bnu}(\bv)=\ct_{\bnu}(\bu^{s})+\bh_2&\mbox{on} \ \ \partial{\Omega},\\[5pt]
\bu^{\mathrm p,\,s}=-\frac{1}{k_{\mathrm p}^2} \nabla (\nabla \cdot \bu^{s} ),\quad \bu^{\mathrm s,\,s}= \frac{1}{k_{\mathrm s} ^2}  \nabla \times  ( \nabla \times   \bu^{s} ) & \mbox{in}\ \ \mathbb{R}^n\backslash\overline{\Omega}, \medskip\\
\lim_{|\bx| \to\infty}|\bx|^{(n-1)/2}\left(  \frac{\partial{\bu^{\mathrm  t,s}}}{\partial{|\bx|}}   -i\kappa_{t}\bu^{\mathrm  t,s} \right)=0,& {\mathrm  t=\mathrm  p,\mathrm s,} \\[5pt]
\end{array}
 \right.
\end{align}
and  find $(\bv,\,\bu^s)\in  H^1(\Omega\backslash\overline{D})^n\times H^1(B_r\backslash\overline{\Omega})^n$ satisfying the following truncated system:
\begin{align}\label{mode:transmisTrans}
\left\{ \begin{array}{ll}
\mathcal{L}_{\mathcal{C}}\bv
+\omega^2\rho(\bx)\bv=\bf 0&\mbox{in}\ \  \Omega\backslash \overline{D},\quad \\[5pt]
\mathcal{L}_{\mathcal{C}^e}\bu^s
+\omega^2\rho_e\bu^s=\bff&\mbox{in}\ \ \mathbb{R}^n\backslash \overline{\Omega},\\[5pt]
\bu^s=\bu^{\mathrm p,\,s}+\bu^{\mathrm s,\,s}&\mbox{in}\ \  B_r\backslash \overline{\Omega},\\[5pt]
\ct_{\bnu}(\bu)=\bp&\mbox{on}\ \  \partial{D},\\[5pt]
\bu\big|_{\partial \Omega}=\bu^{s}\big|_{\partial \Omega}+\bh_1,\quad \ct_{\bnu}(\bu)=\ct_{\bnu}(\bu^{s})+\bh_2&\mbox{on}\ \  \partial{\Omega},\\[5pt]
\bu^{\mathrm p,\,s}=-\frac{1}{k_{\mathrm p}^2} \nabla (\nabla \cdot \bu^{s} ),\quad \bu^{\mathrm s,\,s}= \frac{1}{k_{\mathrm s} ^2}  \nabla \times  ( \nabla \times   \bu^{s} ) & \mbox{in}\ \ B_r\backslash\overline{\Omega}, \medskip\\
\ct_{\bnu}(\bu^s)=\Lambda \bu^s&\mbox{on}\ \ \partial{B_r},
\end{array}
 \right.
\end{align}
where $\Lambda$ is the  Dirichlet-to-Newmann (DtN) map introduced in \cite{Charalambopoulos03,Isakov2006} such that
\begin{align}\label{defin:DtN}
\Lambda:H^{1/2}(\partial B_r)^n&\longrightarrow H^{-1/2}(\partial B_r)^n, \nonumber\\
\widetilde{\bg}&\longmapsto \ct_{\bnu}(\widetilde{\bq})
\end{align}
with a radiating solution $\widetilde{\bq}$ for Navier equation
$$
\left\{
\begin{array}{ll}
\mu_e\Delta\widetilde{\bq}+(\lambda_e+\mu_e)\nabla( \nabla\cdot\widetilde{\bq})+\omega^2\,\rho_e\widetilde{\bq}=\bf 0 &\mbox{in }  \quad \mathbb{R}^n\backslash\overline{B_r}, \\
\widetilde{\bq}=\widetilde{\bg}& \mbox{on}  \quad \partial B_r,
\end{array}
\right.
$$
where $\lambda_e, \mu_e$ and $\rho_e$ are real constants satisfying the strong convexity condition \eqref{eq:convexity1}.

In the following, we establish the equivalence of problem \eqref{mode:transmission} and problem \eqref{mode:transmisTrans} in Lemma \ref{Pro:transmisTranss}. Therefore we can prove that   \eqref{mode:transmission} admits a unique solution and satisfies certain a priori estimates.
\begin{lem}\label{Pro:transmisTranss}
The scattering problems \eqref{mode:transmission} and \eqref{mode:transmisTrans} are equivalent.
\end{lem}
\begin{proof}
By applying the definition of $\Lambda$, it is easy to see that if $(\bv,\bu^s)$ is a solution to the scattering problem \eqref{mode:transmission}, then $ (\bv,\bu^s)\big|_{B_r\backslash\overline{D}}$ solves the scattering problem \eqref{mode:transmisTrans}.

On the other hand, suppose  $(\bv,\bu^s)$ is a solution to the truncated system \eqref{mode:transmisTrans}. By applying the integral representation  and $\ct_{\bnu}\,\bu^s=\Lambda\bu^s$ on $\partial B_r$, we can derive that
\begin{align}\label{eq:Green's representation}
\bu^s(\bx)=&\int_{\partial B_r}\Big\{ \big\{\ct^\by_{\bnu}\Phi(\bx,\by)\}^\top\cdot\bu^s(\by)-\Phi(\bx,\by)\cdot \Lambda \bu^s(\by)\Big\}\,\rmd s(\by)+\int_{B_r\backslash\overline{\Omega}}\Phi(\bx,\by)\cdot\bff(\by)\,\rmd\by\nonumber\\
&-\int_{\partial\Omega}\Big\{\big\{\ct^\by_{\bnu}\Phi(\bx,\by)\big\}^\top\cdot\bu^s(\by)
-\Phi(\bx,\by)\cdot\ct^\by_{\bnu}\bu^s(\by)\Big\}\,\rmd s(\by),
\end{align}
where $\Phi(\bx,\by)$ is the fundamental solution to the Lam\'{e} system \eqref{eq:incident1} with the form
\begin{align}\label{eq:fundamental}
\Phi(\bx,\by)=&\dfrac{\kappa_{\mathrm s}^2}{4\pi\omega^2}\dfrac{e^{\imath\kappa_{\mathrm s}|\bx-\by|}}{|\bx-\by|}I +\dfrac{1}{4\pi\omega^2}\nabla_{\bx}\nabla_{\bx}^\top\dfrac{e^{\imath\kappa_{\mathrm s} |\bx-\by|}-e^{\imath\kappa_{\mathrm p}|\bx-\by|}}{|\bx-\by|}
\end{align}
and
\begin{align}\label{eq:fundamentalTran}
\ct^{\by}_{\bnu}\Phi(\bx,\by)=&\Big[\ct^{\by}_{\bnu}
\Big(\Phi(\bx,\by)(:,1)\Big),\,\ct^{\by}_{\bnu}\Big(\Phi(\bx,\by)(:,2)\Big),\,\ct^{\by}_{\bnu}\Big(\Phi(\bx,\by)(:,3)\Big)\Big].
 \end{align}
Here, $I$ is the identity matrix, $\Phi(\bx,\by)(:,j)$ denotes the $j$-th column of $\Phi(\bx,\by)$, $j=1,2,3$. $\ct^{\by}_{\bnu}$ is the exterior unit normal vector to the boundaries with respect to $\by$. Notice that $\Phi(\bx,\by)=\Phi(\bx,\by)^{\top}$. Then, by combining the definition of $\Lambda$ with the fact that each column of $\Phi(\bx,\by)$ satisfies the Kupradze radiation condition, we can obtain that
\begin{align}\label{eq:Betti 2 identity}
\int_{\partial B_r}\Big\{ \big\{\ct^\by_{\bnu}\Phi(\bx,\by)\big\}^\top\cdot\bu^s(\by)-\Phi(\bx,\by)\cdot \Lambda \bu^s(\by)\Big\}\rmd s(\by)=0.
\end{align}
Substituting \eqref{eq:Betti 2 identity} into \eqref{eq:Green's representation} yields
\begin{align}\notag 
\bu^s(\bx)=&-\int_{\partial\Omega}\Big\{\big\{\ct^\by_{\bnu}\Phi(\bx,\by)\big\}^\top\cdot\bu^s(\by)
-\Phi(\bx,\by)\cdot\ct^\by_{\bnu}\bu^s(\by)\Big\}\,\rmd s(\by) \nonumber \\
&+\int_{B_r\backslash\overline{\Omega}}\Phi(\bx,\by)\cdot\bff(\by)\,\rmd\by. \notag
\end{align}
Clearly, $\bu^s$ can be extended to a function belong to $H^1_{loc}(\mathbb{R}^n\backslash \overline{\Omega})^n$ (still denoted by $\bu^s$). Since each column of $\Phi(\bx,\by)$ or $\ct^{\by}_{\bnu}\Phi(\bx,\by)$ satisfies the Kupradze radiation condition, the new function $\bu^s\in H^1_{loc}(\mathbb{R}^n\backslash \overline{\Omega})^n$ also satisfies the Kupradze radiation condition. Hence, $(\bv,\bu^s)$ solves problem \eqref{mode:transmission}.
\end{proof}

In next lemma, we prove that there exists a unique solution to the system \eqref{mode:transmission}, which is determined by the inputs $\bp$, $\bh_1$, $\bh_2$ and $\bff$.
\begin{lem}\label{lem:unique proof}
Given $\bp\in H^{-1/2}(\partial D)^n$, $\bh_1\in H^{1/2}(\partial \Omega)^n$, $\bh_2\in H^{-1/2}(\partial \Omega)^n$ and $\bff$ with $\supp(\bff)\subset B_{r_0}\backslash\overline{\Omega}$, there exists a unique solution $ (\bv,\,\bu^s)\in H^1(\Omega\backslash\overline{D})^n\times H^1(\mathbb{R}^3\backslash\overline{\Omega})^n$ to the system \eqref{mode:transmission}  such that  the following estimate holds
\begin{align}\label{ineq:EstiTransmission}
\|\bv\|_{H^1(\Omega\backslash\overline{D})^n}+\|\bu^s\|_{H^1(\mathbb{R}^n\backslash\overline{\Omega})^n}&\leq C\bigg( \|\bp\|_{H^{-1/2}(\partial D)^n }+ \|\bh_1\|_{H^{1/2}(\partial \Omega)^n} \nonumber\\
&\quad+\|\bh_2\|_{ H^{-1/2}(\partial \Omega)^n}  +\|\bff\|_{L^2(B_{r_0}\backslash\overline{\Omega})^n}                                    \bigg)
\end{align}
for some constant $C>0$ depending only on $\Ccal(\bx)$, $\kappa_p$, $\kappa_s$, $\Ccal^e$, $\rho(\bx)$, $\Omega$, $D$, $B_r$ and $\omega$.
\end{lem}
\begin{proof}
Firstly, let $\bp=\mathbf 0,\,\bh_1=\mathbf 0,\,\bh_2=\mathbf 0,\,\bff=\mathbf 0$. It is sufficient to show that there exists only a  trivial solution to \eqref{mode:transmission}.  Post-multiplying the first  equation of \eqref{mode:transmission}, respectively, by $\overline{\bv}$ and $\overline{\bu^s}$ and  using the Betti's first formula (cf. \cite{Alves2002,Ciarelet2005,Ciarlet1988})
over $\Omega\backslash\overline{D}$ and $B_r\backslash\overline{\Omega}$ and the boundary conditions on $\partial D$ and $\partial \Omega$, we have
\begin{align}\label{eq:integra1}
\int_{\Omega\backslash\overline{D}}\,\,[\Ccal(\bx):\nabla \overline{\bv}]:\nabla \bv \,\rmd\bx&=\int_{\Omega\backslash\overline{D}}\,\omega^2\rho\,|\bv|^2\,\rmd\bx
-\int_{B_r\backslash\overline{\Omega}}\,\,[\Ccal^e:\nabla\overline{\bu^s}]:\nabla \bu^s \,\rmd s(\bx)\nonumber\\
&\,+\int_{\partial B_r}\bnu\cdot[\Ccal^e:\nabla \bu^s]\cdot\overline{\bu^s}\,\rmd s(\bx)+\int_{B_r\backslash\overline{\Omega}}\,\omega^2\,\rho_e\,|\bu^s|^2\,\rmd\bx.
\end{align}
Taking the imaginary part of the equation above, we obtain
$$
\Im\int_{\partial B_r}\bnu\cdot[\Ccal^e:\nabla \bu^s]\,\,\,\overline{\bu^s}\,\rmd s(\bx)=-\int_{\Omega\backslash\overline{D}}\omega^2\,\Im \rho\,|\bu|^2 \,\rmd\bx\leq0.
$$
From Lemma \ref{Rellich} and the unique continuation principle, we know $\bu^s={\bf 0}$ in $\Omega\backslash\overline{D}$ and $\bv={\bf 0}$ in $D$. Therefore, the uniqueness of the solution to \eqref{mode:transmission} is established.

By Lemma \ref{Pro:transmisTranss}, problems \eqref{mode:transmission} and \eqref{mode:transmisTrans} are equivalent. Thus, we only need to verify the existence of solution to \eqref{mode:transmisTrans} by the variational technique.
Without  loss of generality, we assume $\omega^2\rho_e$ is not a Dirichlet eigenvalue in $B_r\backslash\overline{\Omega}$.
It is easy to check that the vector field $\bw$, which is  defined by $\bw(\bx)=\bv(\bx)$ in $\Omega\backslash\overline{D}$ and $\bw(\bx)=\bu^s(\bx)+\widetilde{\bv}(\bx)$ in $B_r\backslash\overline{\Omega}$, satisfies
\begin{align}\label{mode:AuxiCom}
\left\{ \begin{array}{ll}
\mathcal{L}_{\mathcal{C}}\bw
+\omega^2\rho(\bx)\bw=\bff&\mbox{in}\ \ B_r\backslash \overline{D},\\[5pt]
\bw^s=\bw^{\mathrm p,\,s}+\bw^{\mathrm s,\,s}&\mbox{in}\ \  B_r\backslash \overline{\Omega},\\[5pt]
\ct_{\bnu}(\bw)=\bp&\mbox{on}\ \  \partial{D},\\[5pt]
\bw^-=\bw^+&\mbox{on}\ \ \partial \Omega,\\[5pt]
\ct_{\bnu}(\bw^-)=\ct_{\bnu}(\bw^+)+\ct_{\bnu}(\bu^{in})-\ct_{\bnu}(\widetilde{\bv})&\mbox{on}\ \  \partial \Omega,\\[5pt]
\ct_{\bnu}(\bw^-)=\Lambda\bw^{+}+\ct_{\bnu}(\widetilde{\bv})&\mbox{on}\ \ \partial B_r,\\[5pt]
\end{array}
 \right.
\end{align}
where  $\bw^{-}$ and $\bw^{+}$ stand for the limits from outside and inside $\partial \Omega$, respectively, $\Lambda$ is the DtN operator given in \eqref{defin:DtN}, $\widetilde{\bv}$ is a solution to the following equation:
\begin{align}\label{eq:DtN01}
\left\{ \begin{array}{ll}
\mu_e \Delta \widetilde{\bv}+(\lambda_e+\mu_e) \nabla(\nabla\cdot \widetilde{\bv})+\omega^2\rho_e\widetilde{\bv}=\bf 0 &\mbox{in}\ \ B_r\backslash\overline{\Omega}\\
 \widetilde{\bv}=\bu^{in}&\mbox{on}\ \ \partial\Omega,\\
 \widetilde{\bv}={\bf 0}&\mbox{on}\ \ \partial B_r.
\end{array}
 \right.
\end{align}
By \cite[Theorem 4.10]{Mclean2000}, we know that $\widetilde{\bv}$ is unique and $\|\widetilde{\bv}\|_{H^1(B_r\backslash\overline{\Omega})^n}=O(\|\bu^{in}\|_{H^{1/2}(\partial\Omega)^n})$.

Next, we introduce a bounded operator
$$
\Lambda_0:H^{1/2}(\partial B_r)^n\longrightarrow H^{-1/2}(\partial B_r)^n
$$
which maps $\Phi$ to $\ct_{\bnu}(\widetilde{\bw})\bigg|_{\partial B_r}$
where $\widetilde{\bw}\in H^1_{loc}(\mathbb{R}^n\backslash\overline{B_r})^n$ is the unique solution of the following system:
\begin{align}\label{mode:AuxiCom2}
\left\{ \begin{array}{ll}
\mu_e \Delta \widetilde{\bw}+(\lambda_e+\mu_e) \nabla(\nabla\cdot \widetilde{\bw})+\omega^2\rho_e\widetilde{\bw}=\bf 0&\mbox{in}\quad\mathbb{R}^n\backslash\overline{B_r},\\[5pt]
\widetilde{\bw}=\Phi \in H^{1/2}(\partial B_r)^n&\mbox{on}\quad\partial B_r.
\end{array}
 \right.
\end{align}
The operator $\Lambda_0$ has the following properties
\begin{align}\label{ineq:DtN}
-\int_{\partial  B_r} \overline{\Phi}\Lambda_0\Phi\, ds(x)\geq0,\quad \Phi\in H^{1/2}(\partial B_r)^n,
\end{align}
and the difference $\Lambda-\Lambda_0$ is a compact operator from $H^{1/2}(\partial B_r)^n\rightarrow H^{-1/2}(\partial B_r)^n$.  It is proved in \cite{Charalambopoulos03} that these properties still hold for dyadic field by the similar analysis for the Laplace operator \cite{Cakoni2006,Hahenr2000}.
Hence, for any $\bvarphi\in H^1(B_r\backslash\overline{D})^n$, using the test function $\overline{\bvarphi}$ we can easily derive the variational formulation of  \eqref{mode:AuxiCom}: find $\bw\in  H^{1}(B_r\backslash\overline{D})^n$ such that
\begin{equation}\label{eq:variational formulation}
a_1(\bw, \bvarphi)+a_2(\bw, \bvarphi)=\cf(\bvarphi),
\end{equation}
where the bilinear forms $a_1$, $a_2$ and the linear functional $\cf(\cdot)$ are defined by
\begin{align}
a_1(\bw,\bvarphi)&:=\int_{\Omega\backslash\overline{D}}\,\,(\Ccal(\bx):\nabla\overline{\bvarphi}):\nabla\bw\,\rmd\bx+\int_{\Omega\backslash\overline{D}}\,\rho \,\omega^2\,\bw\cdot \overline{\bvarphi}\,\rmd\bx+\int_{B_r\backslash\overline{\Omega}}\,\,(\Ccal^e:\nabla\overline{\bvarphi}):\nabla\bw\,\rmd\bx\nonumber\\
&\quad+\int_{B_r\backslash\overline{\Omega}}\,\omega^2\,\rho_e\,\bw\cdot\overline{\bvarphi}\,\rmd\bx-\int_{\partial B_r}\,\Lambda_0\bw\cdot\overline{\bvarphi}\,\rmd s(\bx),\notag \\ 
a_2(\bw,\bvarphi) &:=-2\int_{\Omega\backslash\overline{D}}\,\rho\, \omega^2\,\bw\cdot \overline{\bvarphi}\,\rmd\bx -2\int_{B_r\backslash\overline{\Omega}} \omega^2\,\rho_e\,\bw\cdot \overline{\bvarphi}\,\rmd\bx\nonumber\\
&\quad\ -\int_{\partial B_r}(\Lambda-\Lambda_0)\bw\cdot \overline{\bvarphi}\,\rmd s(\bx),\notag \\
\cf(\bvarphi):=&-\int_{\partial D}\,\bp\cdot \overline{\bvarphi}\,\,\rmd s(\bx)+\int_{\partial \Omega}\,(\bh_2-\ct_{\bnu}(\widetilde{\bv}))\cdot \overline{\bvarphi}\,\,\rmd s(\bx)+\int_{\partial B_r}\ct_{\bnu}(\widetilde{\bv})\cdot \overline{\bvarphi}\,\,\rmd s(\bx)\nonumber\\
&-\int_{B_r\backslash\overline{D}}\bff\cdot\overline{\bvarphi}\,\rmd\bx. \notag
\end{align}

By using the assumptions  about $\rho(\bx)$ and $\Ccal(\bx)$ given in Subsection \ref{subsect:motivation}, Cauchy-Schwarz inequality and the definition of operator $\Lambda_0$, one can show the boundedness of the bilinear form $a_1$: for any $ \bphi,\,\bvarphi\in H^{1}(B_r\backslash\overline{D})^n$,
$$\big|a_1(\bphi,\,\bvarphi)\big|\leq C_1\,\|\bphi\|_{H^1(B_r\backslash\overline{D})^n}\|\bvarphi\|_{H^1(B_r\backslash\overline{D})^n}
$$
for some constant $C_1$.
Furthermore, by virtue of Poincar\'e's inequality and \eqref{ineq:DtN}, we have the coercivity property of the bilinear form $a_1$: for any $\bvarphi\in H^{1}(B_r\backslash\overline{D})^n$,
$$
a_1(\bvarphi,\,\bvarphi)\geq C_2\,\|\bvarphi\|^2_{H^1(B_r\backslash\overline{D})^n}
$$
for some constant $C_2$. According to Lax-Milgram lemma, there exists a bounded inverse operator $\cl:H^1(B_r\backslash\overline{D})^n\longrightarrow H^1(B_r\backslash\overline{D})^n$ such that
$$
a_1(\bw,\bvarphi)=\langle\cl\bw, \bvarphi\rangle,
$$
where $\langle\cdot,\cdot\rangle$ is the inner product in $H^1(B_r\backslash\overline{D})^n$, and the inverse of $\cl$ is also bounded. In view of the expression of the bilinear form $a_2$, we introduce two bounded operators $\ck_1$ and $\ck_2$ given by
\begin{align}
\langle\ck_1\bw, \bvarphi\rangle:=&\,2\, \int_{\Omega\backslash\overline{D}}\rho\, \omega^2\,\bw\cdot \overline{\bvarphi}\,\rmd\bx +2\int_{B_r\backslash\overline{\Omega}} \omega^2\,\rho_e\,\bw\cdot \overline{\bvarphi}\,\rmd\bx,\label{eq:A25}\\
\langle\ck_2\bw, \bvarphi\rangle:=&\int_{\partial B_r}(\Lambda-\Lambda_0)\bw^+\cdot \overline{\bvarphi}\,\,\rmd s(\bx).\notag
\end{align}

We claim that the operators $\ck_1$ and $\ck_2$ are both compact. In fact, let $\{\bw_n\}^{\infty}_{n=1}$ be a bounded sequence in $H^1(B_r\backslash\overline{D})^n$ and weakly converge to $\bw_*$ in the sense of $\|\cdot\|_{H^1(B_r\backslash\overline{D})^n}$ (denoted by $\bw_n\rightharpoonup\bw_*$). Since $\ci:H^1(B_r\backslash\overline{D})^n\longrightarrow L^2(B_r\backslash\overline{D})^n$  is a compact embedding operator,
we get
$$
\langle\ck_1(\bw_n-\bw_*),\bvarphi\rangle = 2\int_{\Omega\backslash\overline{D}}\rho\,\omega^2\,(\bw_n-\bw_*)\cdot\overline{\bvarphi} \,\rmd\bx+2\int_{B_r\backslash\overline{\Omega}} \omega^2\,\rho_e\,(\bw_n-\bw_*)\cdot \overline{\bvarphi}\,\rmd\bx
$$
and thus
\begin{align*}
\Big\| \ck_1(\bw_n-\bw_*) \Big\|^2_{H^1(B_r\backslash\overline{D})^n}&=\,\langle\ck_1(\bw_n-\bw_*),\,\ck_1(\bw_n-\bw_*)\rangle\\
&=2\int_{\Omega\backslash\overline{D}}\rho\,\omega^2\,(\bw_n-\bw_*)\cdot\overline{\ck_1(\bw_n-\bw_*)} \,\rmd\bx\\
&\quad+2\int_{B_r\backslash\overline{\Omega}} \rho_e\,\omega^2\,(\bw_n-\bw_*)\cdot \overline{\ck_1(\bw_n-\bw_*)}\,\rmd\bx\\
&\leq 2\,C\,\omega^2\,\max\big\{\|\rho(\bx)\|_{\bL^{\infty}(\Omega\backslash\overline{D})},\rho_e\big\}\,\big\| \bw_n-\bw_* \big\|^2_{L^2(B_r\backslash\overline{D})^n},
\end{align*}
which implies that $\ck_1$ is compact.
Similarly, we can verify the compactness of $\ck_2$. Since $\bw_n\rightharpoonup\bw_*$ in $H^1(B_r\backslash\overline{D})^n$, we have
$\bw_n\big|_{\partial B_r}\rightharpoonup\bw_*\big|_{\partial B_r}$ in $H^{1/2}(\partial B_r)^n$ by the trace operator. Together with the compactness of $\Lambda-\Lambda_0$, it is easy to obtain that
$$
(\Lambda-\Lambda_0)\bw_n\big|_{\partial B_r}\longrightarrow(\Lambda-\Lambda_0)\bw_*\big|_{\partial B_r}
$$
in $H^{-1/2}(\partial B_r)^n$. For any $\bvarphi\in H^1(B_r\backslash\overline{D})^n$, it holds that
$$
\langle\ck_2(\bw_n-\bw_*),\bvarphi\rangle=\int_{\partial B_r}(\Lambda-\Lambda_0)(\bw_n-\bw_*)\cdot \overline{\bvarphi}\,\,\rmd s(x).
$$
Therefore we have
\begin{align*}
\Big\| \ck_2(\bw_n-\bw_*) \Big\|^2_{H^1(B_r\backslash\overline{D})^n}&=\,\left\langle\ck_2(\bw_n-\bw_*),\ck_2(\bw_n-\bw_*)\right\rangle\\
&=\int_{\partial B_r}(\Lambda-\Lambda_0)(\bw_n-\bw_*)\cdot \overline{\ck_2(\bw_n-\bw_*)}\,\,\rmd s(x)\\
&\leq\|(\Lambda-\Lambda_0)(\bw_n-\bw_*)\|_{H^{-1/2}(\partial B_r)^n}\|\ck_2(\bw_n-\bw_*)\|_{H^{1/2}(\partial B_r)^n}\\
&\leq C\,\big\|(\Lambda-\Lambda_0)(\bw_n-\bw_*)\big\|_{H^{-1/2}(\partial B_r)^n}\,\big\| \bw_n-\bw_* \big\|_{L^2(B_r\backslash\overline{D})^n},
\end{align*}
which implies that $\ck_2$ is compact.

Since $\cl$ is bounded and $\ck_1+\ck_2$ is compact, we know  that $\cl-(\ck_1+\ck_2)$ is a Fredholm operator of index zero. According to the Fredholm alternative theorem, Riesz representation theory and the uniqueness of \eqref{mode:transmission}, we know there must exist a solution to \eqref{mode:transmission}. Since the inverse of $\cl-(\ck_1+\ck_2)$ is bounded, by applying the Lax-Milgram lemma to
$$\big\langle\left(\ct -\ck_1-\ck_2\right)\bw,\bvarphi \big\rangle=\cf(\bvarphi),
$$
we get
$$\|\bw\|_{H^1(B_R\backslash\overline{D})^n}\leq C\|\cf\|.
$$

On the other hand, it is straightforward to verify that
\begin{small}
$$\big|\cf(\bvarphi)\big|\leq C\,\Big(\|\bp\|_{H^{-1/2}(\partial D)^n}+\|\bh_2\|_{H^{-1/2}(\partial D)^n})+\|\bh_1\|_{H^{1/2}(\partial \Omega)^n}+\|\bff\|_{H^{-1/2}(B_r\backslash\overline{D})^n}\Big)\,\|\bvarphi\|_{H(B_r\backslash\overline{D})^n},
$$
\end{small}
which can directly imply the inequality \eqref{ineq:EstiTransmission}.
\end{proof}

\subsection{Auxiliary lemmas for Case 2}\label{subsec:case 2}
In this subsection, we shall establish several key lemmas  for Case 2 in Theorem \ref{thm:main1}. Considering that $D$ is a rigid obstacle, the unbounded and truncated scattering systems associated with Case 2 are given as follows:
find $(\bv,\,\bu^s)\in H^1(\Omega\backslash\overline{D})^n$ $\times H^1(\mathbb{R}^n\backslash\overline{\Omega})^n$ satisfying
\begin{align}\label{mode:transmission 2}
\left\{ \begin{array}{ll}
\mathcal{L}_{\mathcal{C}}\bv
+\omega^2\rho(\bx)\bv=\bf 0&\mbox{in}\ \  \Omega\backslash \overline{D},\quad \\[5pt]
\mathcal{L}_{\mathcal{C}^e}\bu^s
+\omega^2\rho_e\bu^s=\bff&\mbox{in}\ \ \mathbb{R}^n\backslash \overline{\Omega},\\[5pt]
\bu^s=\bu^{p,\,s}+\bu^{s,\,s}&\mbox{in}\ \ \mathbb{R}^n\backslash \overline{\Omega},\\[5pt]
\bu\big|_{\partial D}=\bp&\mbox{on}\ \ \partial{D},\\[5pt]
\bv=\bu^{s}+\bh_1,\quad \ct_{\bnu}(\bv)=\ct_{\bnu}(\bu^{s})+\bh_2&\mbox{on} \ \ \partial{\Omega},\\[5pt]
\bu^{\mathrm p,\,s}=-\frac{1}{k_{\mathrm p}^2} \nabla (\nabla \cdot \bu^{s} ),\quad \bu^{\mathrm s,\,s}= \frac{1}{k_{\mathrm s} ^2}  \nabla \times  ( \nabla \times   \bu^{s} ) & \mbox{in}\ \ \mathbb{R}^n\backslash\overline{\Omega}, \medskip\\
\lim_{|\bx| \to\infty}|\bx|^{(n-1)/2}\left(  \frac{\partial{\bu^{\mathrm  t,s}}}{\partial{|\bx|}}   -i\kappa_{t}\bu^{\mathrm  {t},s}  \right)=0,& {\mathrm t= t,s}\\[5pt]
\end{array}
 \right.
\end{align}
and  find $(\bv,\,\bu^s)\in  H^1(\Omega\backslash\overline{D})^n\times H^1(B_r\backslash\overline{\Omega})^n$ satisfying
\begin{align}\label{mode:transmisTrans 2}
\left\{ \begin{array}{ll}
\mathcal{L}_{\mathcal{C}}\bv
+\omega^2\rho(\bx)\bv=\bf 0&\mbox{in}\ \  \Omega\backslash \overline{D},\quad \\[5pt]
\mathcal{L}_{\mathcal{C}^e}\bu^s
+\omega^2\rho_e\bu^s=\bff&\mbox{in}\ \ \mathbb{R}^n\backslash \overline{\Omega},\\[5pt]
\bu\big|_{\partial D}=\bp&\mbox{on}\ \  \partial{D},\\[5pt]
\bv\big|_{\partial \Omega}=\bu^{s}\big|_{\partial \Omega}+\bh_1,\quad \ct_{\bnu}(\bv)=\ct_{\bnu}(\bu^{s})+\bh_2&\mbox{on}\ \  \partial{\Omega},\\[5pt]
\bu^{\mathrm p,\,s}=-\frac{1}{k_{\mathrm p}^2} \nabla (\nabla \cdot \bu^{s} ),\quad \bu^{\mathrm s,\,s}= \frac{1}{k_{\mathrm s} ^2}  \nabla \times  ( \nabla \times   \bu^{s} ) & \mbox{in}\ \ B_r\backslash\overline{\Omega}, \medskip\\
\ct_{\bnu}(\bu^s)=\Lambda \bu^s&\mbox{on}\ \ \partial{B_r},
\end{array}
 \right.
\end{align}
where $\bp\in H^{1/2}(\partial D)^n$, $\bh_1\in H^{1/2}(\partial \Omega)^n$, $\bh_2\in H^{-1/2}(\partial \Omega)^n$ and $\bff $ with $\supp(\bff)\subset B_{r_0}\backslash\overline{\Omega}\subset B_{r}\backslash\overline{\Omega}$. In fact, we can easily obtain the equivalence of \eqref{mode:transmission 2} and\eqref{mode:transmisTrans 2} by the similar argument of Lemma \ref{Pro:transmisTranss}.
In addition, similar to Lemma \ref{lem:unique proof}, we have the following result.
\begin{lem}\label{lem:unique proof 2}
Given $\bp\in H^{1/2}(\partial D)^n$, $\bh_1\in H^{1/2}(\partial \Omega)^n$, $\bh_2\in H^{-1/2}(\partial \Omega)^n$ and $\bff$ with $\supp(\bff)\subset B_{r_0}\backslash\overline{\Omega}$, there exists a unique solution $ (\bv,\,\bu^s)\in H^1(\Omega\backslash\overline{D})^n\times H^1(\mathbb{R}^3\backslash\overline{\Omega})^n$ to the system \eqref{mode:transmission} such that  the following estimate holds
\begin{align}\label{ineq:EstiTransmission}
\|\bv\|_{H^1(\Omega\backslash\overline{D})^n}+\|\bu^s\|_{H^1(\mathbb{R}^n\backslash\overline{\Omega})^n}&\leq C\bigg( \|\bp\|_{H^{1/2}(\partial D)^n }+ \|\bh_1\|_{H^{1/2}(\partial \Omega)^n} \nonumber\\
&\quad+\|\bh_2\|_{ H^{-1/2}(\partial \Omega)^n}  +\|\bff\|_{L^2(B_{r_0}\backslash\overline{\Omega})^n}                                    \bigg),
\end{align}
where $C$ is a positive constant.
\end{lem}
\begin{proof}
The uniqueness can be easily proved. Thus,  we only need to verify that there exists a solution to \eqref{mode:transmisTrans 2}  such that the  estimate \eqref{ineq:EstiTransmission} holds. Without  loss of generality, we assume $\omega^2\,\rho_e$ is not a Dirichlet eigenvalue in $B_r\backslash\overline{\Omega}$. The PDE system \eqref{mode:transmisTrans 2} can be converted into the following system:
\begin{align}\label{mode:AuxiCom 2}
\left\{ \begin{array}{ll}
\mathcal{L}_{\mathcal{C}}\bw
+\omega^2\rho(\bx)\bw=\bff&\mbox{in} \ \ B_r\backslash \overline{D},\\[5pt]
\bw^s=\bw^{\mathrm p,\,s}+\bw^{\mathrm s,\,s}&\mbox{in}\ \  B_r\backslash \overline{\Omega},\\[5pt]
\bw\big|_{\partial D}=\bp&\mbox{on}\ \  \partial{D},\\[5pt]
\bw^-\big|_{\partial \Omega}=\bw^+\big|_{\partial \Omega}&\mbox{on}\ \ \partial \Omega,\\[5pt]
\ct_{\bnu}(\bw^-)=\ct_{\bnu}(\bw^+)+\ct_{\bnu}(\bu^{in})-\ct_{\bnu}(\widetilde{\bv})&\mbox{on}\ \  \partial \Omega,\\[5pt]
\bw^{\mathrm p,\,s}=-\frac{1}{k_{\mathrm p}^2} \nabla (\nabla \cdot \bu^{s} ),\quad \bw^{\mathrm s,\,s}= \frac{1}{k_{\mathrm s} ^2}  \nabla \times  ( \nabla \times   \bu^{s} ) & \mbox{in}\ \ B_r\backslash\overline{\Omega}, \medskip\\
\ct_{\bnu}(\bw)=\Lambda\bw+\ct_{\bnu}(\widetilde{\bv})&\mbox{on}\ \ \partial B_r,\\[5pt]
\end{array}
 \right.
\end{align}
where $\bw(\bx)=\bv(\bx)$ in $\Omega\backslash\overline{D}$ and $\bw(\bx)=\bu^s(\bx)+\widetilde{\bv}(\bx)$ in $B_r\backslash\overline{\Omega}$, $\widetilde{\bv}$ is a solution to
$$
\begin{cases}
	\mathcal{L}_{\mathcal{C}^e}  \bw
+\omega^2\rho_e  \bw=\bff \mbox{ in } B_r\backslash \overline{\Omega}, \\
\widetilde{\bv}=\bu^{in} \hspace{1.9cm} \mbox{ on } \partial\Omega,\\
\widetilde{\bv}={\bf 0} \hspace{2.2cm} \mbox{ on } \partial B_r.
\end{cases}
$$
By \cite[Theorem 4.10]{Mclean2000}, it is obvious to see that $\widetilde{\bv}$ is unique and
$$
\|\widetilde{\bv}\|_{H^1(B_r\backslash\overline{\Omega})^n}=O(\|\bu^{in}\|_{H^{1/2}(\partial\Omega)^n}).
$$ Similar to the proof of Lemma \ref{lem:unique proof 2}, we also  use the bounded operator $\Lambda_0$ and its corresponding properties. Here, we introduce a new Sobolve space
$$
X:=\{\bw\in H^{1}(B_r\backslash
\overline{D})^n;\bw={\bf 0}\ \ \mbox{on}\ \partial D\}
$$
and  let $\bw_0\in H^1(B_r\backslash\overline{D})^n$ be such that $\bw_0=\bp$ on $\partial D$ and $\|\bw_0\|_{H^1(B_r\backslash\overline{D})^n}\leq C\|\bq\|_{H^{1/2}(\partial D)^n}$.
Then for any $\bvarphi\in X$, using the test function $\overline{\bvarphi}$ we can easily derive the variational formulation of \eqref{mode:AuxiCom 2}: find $\bw\in  H^{1}(B_r\backslash\overline{D})^n$ such that
\begin{equation}\label{eq:variational formulation}
a_1(\bw-\bw_0, \bvarphi)+a_2(\bw-\bw_0, \bvarphi)=\cf(\bvarphi),
\end{equation}
where the bilinear forms $a_1$, $a_2$ and the linear functional $\cf(\cdot)$ are defined by
\begin{eqnarray*}
a_1(\bw-\bw_0,\bvarphi)&:=&\int_{\Omega\backslash\overline{D}}\,\,(\Ccal(\bx):\nabla\overline{\bvarphi}):\nabla(\bw-\bw_0)\,\rmd x+\int_{\Omega\backslash\overline{D}}\,\rho \,\omega^2\,(\bw-\bw_0)\cdot \overline{\bvarphi}\,\rmd x\nonumber\\
&&+\int_{B_r\backslash\overline{\Omega}}\,\,(\Ccal^e:\nabla\overline{\bvarphi}):\nabla(\bw-\bw_0)\,\rmd x
+\int_{B_r\backslash\overline{\Omega}}\,\omega^2\,\rho_e\,(\bw-\bw_0)\cdot\overline{\bvarphi}\,\rmd x\nonumber\\
&&-\int_{\partial B_r}\,\Lambda_0(\bw-\bw_0)\cdot\overline{\bvarphi}\,\rmd s(x),\label{eq:a1}\\
a_2(\bw-\bw_0,\bvarphi)&:=&-2\int_{\Omega\backslash\overline{D}}\,\rho\, \omega^2\,(\bw-\bw_0)\cdot \overline{\bvarphi}\,\rmd x -2\int_{B_r\backslash\overline{\Omega}} \omega^2\,\rho_e\,(\bw-\bw_0)\cdot \overline{\bvarphi}\,\rmd x\nonumber\\
&&-\int_{\partial B_r}(\Lambda-\Lambda_0)(\bw-\bw_0)\cdot \overline{\bvarphi}\,\rmd s(x),\\
\cf(\bvarphi)&:=&\int_{\partial \Omega}\,(\bh_2-\ct_{\bnu}(\widetilde{\bv}))\cdot \overline{\bvarphi}\,\,ds(x)+\int_{\partial B_r}\ct_{\bnu}(\widetilde{\bv})\cdot \overline{\bvarphi}\,\,\rmd s(x)-\int_{B_r\backslash\overline{D}}\bff\cdot\overline{\bvarphi}\,\rmd x.
\end{eqnarray*}

Since $\Lambda$ is a bounded operator from $H^{1/2}(\partial B_r)$ to $H^{1/2}(\partial B_r)$, $\cf$ is a bounded conjugate linear functional on $X$ and both $a_1(\cdot,\cdot)$ and $a_2(\cdot,\cdot)$ are continuous on $X\times X$: for any $ \bphi,\,\bvarphi\in X$,
$$\big|a_1(\bphi,\,\bvarphi)\big|\leq C\,\|\bphi\|_{H^1(B_r\backslash\overline{D})^n}\|\bvarphi\|_{H^1(B_r\backslash\overline{D})^n}
$$
for some constant $C$.

From the properties of $\Lambda_0$ and \eqref{eq:ellip1}, we see that for any $\bvarphi\in X$,
$$
a_1(\bvarphi,\bvarphi)\geq C\|\bvarphi\|^2_{H^1(B_r\backslash\overline{D})^n}
$$
with some constant $C$.
Therefore, by Lax-Milgram lemma, there exists a bounded inverse operator $\cl:X\longrightarrow X$ such that
$$
a_1(\bw-\bw_0,\bvarphi)=\langle\cl(\bw-\bw_0), \bvarphi\rangle,
$$
where $\langle\cdot,\cdot\rangle$ is the inner product in $H^1(B_r\backslash\overline{D})^n$, and the inverse of $\cl$ is also bounded. Note that including a $L^2$-inner product term in $a_1(\cdot,\cdot)$ is important since the Poincar\'{e} inequality does not hold in $X$ any longer. From the expression of the bilinear form $a_2$, we introduce two bounded operators $\ck_1$ and $\ck_2$ given by
\begin{align}\label{eq:A25}
\langle\ck_1(\bw-\bw_0), \bvarphi\rangle:=&2\int_{\Omega\backslash\overline{D}}\rho\, \omega^2\,(\bw-\bw_0)\cdot \overline{\bvarphi}\,\rmd x +2\int_{B_r\backslash\overline{\Omega}} \omega^2\,\rho_e\,(\bw-\bw_0)\cdot \overline{\bvarphi}\,\rmd x,\\
\langle\ck_2(\bw-\bw_0), \bvarphi\rangle:=&\int_{\partial B_r}(\Lambda-\Lambda_0)(\bw-\bw_0)\cdot \overline{\bvarphi}\,\,\rmd s(x).\label{eq:A25-2}
\end{align}
By the similar argument as in Lemma \ref{lem:unique proof}, we can verify that the operators $\ck_1$ and $\ck_2$ defined by \eqref{eq:A25} and \eqref{eq:A25-2} are also compact. Similarly, we also have
$$\big\langle\left(\ct -\ck_1-\ck_2\right)(\bw-\bw_0),\bvarphi \big\rangle=\cf(\bvarphi),
$$
By Lax-Milgram lemma, we see that
$$\|(\bw-\bw_0)\|_{H^1(B_R\backslash\overline{D})^n}\leq C\|\cf\|.
$$
On the other hand,
$$\big|\cf(\bvarphi)\big|\leq C\,\Big(\|\bp\|_{H^{1/2}(\partial D)^n}+\|\bh_2\|_{H^{-1/2}(\partial D)^n})+\|\bh_1\|_{H^{1/2}(\partial \Omega)^n}+\|\bff\|_{H^{-1/2}(B_r\backslash\overline{D})^n}\Big)\,\|\bvarphi\|_{H(B_r\backslash\overline{D})^n},
$$
which can directly imply the inequality
\begin{align*}
\|\bv\|_{H^1(\Omega\backslash\overline{D})^n}+\|\bu^s\|_{H^1(\mathbb{R}^n\backslash\overline{\Omega})^n}\leq& C\Big\{ \|\bw_0\|_{H^1(\mathbb{R}^n\backslash\overline{D})^n}+ \|\bp\|_{H^{1/2}(\partial D)^n}+\|\bh_2\|_{H^{-1/2}(\partial D)^n}\\
&\quad\quad\quad\quad\quad\quad\quad
 +\|\bh_1\|_{H^{1/2}(\partial \Omega)^n}+\|\bff\|_{\bH^{-1/2}(B_r\backslash\overline{D})^n}  \Big\}\\
 &\leq \widetilde{C}\Big( \|\bp\|_{H^{1/2}(\partial D)^n }+ \|\bh_1\|_{H^{1/2}(\partial \Omega)^n}
+\|\bh_2\|_{ H^{-1/2}(\partial \Omega)^n}\nonumber\\  &\hspace{6.05cm}+\|\bff\|_{L^2(B_{r_0}\backslash\overline{\Omega})^n}                                    \Big).
\end{align*}
The proof is complete.
\end{proof}

\subsection{The well-posedness of the scattering problem  \eqref{eq:scattering1}}\label{well-posedness} In this subsection, we can adopt a similar variational technique used in Subsections \ref{subsect:auxiliary} and  \ref{subsec:case 2} to verify the well-posedness of the scattering problem \eqref{eq:scattering1}.
\begin{prop}\label{pro:Well-posed}
There exits a unique solution $\bu\in H^1(\mathbb{R}^n\backslash\overline{D})^n$ to the scattering problem \eqref{eq:scattering1}. Furthermore, it holds that
\begin{align}\label{eq:230 bound}
\|\bu\|_{H^1(\mathbb{R}^n\backslash\overline{D})^n}&\leq C\bigg( \|\bu^{in}\|_{H^{1/2}(\partial \Omega)^n} +\|\ct_{\nu}(\bu^{in})\|_{ H^{-1/2}(\partial \Omega)^n}  +\|\bff\|_{L^2(B_{r_0}\backslash\overline{\Omega})^n}                                    \bigg),
\end{align}
where $C$ is a positive constant, $\Omega \Subset B_{r_0} $ and $B_{r_0}$ is a ball centered at the origin with the radius $r_0 \in \mathbb R_+ $.
\end{prop}
\begin{proof}
As discussed in Subsection \ref{subsect:summary of results}, by using an appropriate truncation we can truncate the unbounded domain $\mathbb{R}^n\backslash\overline{D}$ in \eqref{eq:scattering1} into a bounded one. Indeed,  \eqref{eq:scattering1} can be transformed to  the following PDE system:
Find $\bu\in  H^1(B_r\backslash\overline{D})^n$ such that
\begin{align}\label{mode:transmisTrans 0}
\left\{ \begin{array}{ll}
\mathcal{L}_{\mathcal{C}}\bu
+\omega^2\rho(\bx)\bu=\bf 0&\mbox{in}\ \  \Omega\backslash \overline{D},\quad \\[5pt]
\mathcal{L}_{\mathcal{C}^e}\bu^s
+\omega^2\rho_e\bu^s=\bff&\mbox{in}\ \ \mathbb{R}^n\backslash \overline{\Omega},\\[5pt]
\bu^s=\bu^{\mathrm p,\,s}+\bu^{\mathrm s,\,s}&\mbox{in}\ \  B_r\backslash \overline{\Omega},\\[5pt]
\mathcal{B}(\mathbf{u})=\bf 0&\mbox{on}\ \  \partial{D},\\[5pt]
\bu\big|_{\partial \Omega}=\bu^{s}\big|_{\partial \Omega}+\bu^{in},\quad \ct_{\bnu}(\bu)=\ct_{\bnu}(\bu^{s})+\ct_{\bu^{in}}&\mbox{on}\ \  \partial{\Omega},\\[5pt]
\bu^{\mathrm p,\,s}=-\frac{1}{k_{\mathrm p}^2} \nabla (\nabla \cdot \bu^{s} ),\quad \bu^{\mathrm s,\,s}= \frac{1}{k_{\mathrm s} ^2}  \nabla \times  ( \nabla \times   \bu^{s} ) & \mbox{in}\ \ B_r\backslash\overline{\Omega}, \medskip\\
\ct_{\bnu}(\bu^s)=\Lambda \bu^s&\mbox{on}\ \ \partial{B_r}.
\end{array}
 \right.
\end{align}
In fact, we can use a completely similar argument of Lemma \ref{Pro:transmisTranss} to verify the equivalence of \eqref{mode:transmisTrans 0} and \eqref{eq:scattering1}, so we only need to illustrate that there exists a unique solution to \eqref{eq:scattering1} and it is relied on the input data $\bu^{in}$ and $\bff$. Here we replace the product space $ H^1(\Omega\backslash\overline{D})^n\times H^1(\mathbb{R}^n\backslash\overline{\Omega})^n$ in Lemma
\ref{lem:unique proof} (or $X\times X$ in Lemma \ref{lem:unique proof 2}) with $\mathbb{R}^n\backslash\overline{D}$  and take $\mathcal{B}(\mathbf{u}) =\bf 0$ on $\partial D$, $\bh_1= \bu^{in}\in H^{1/2}(\partial \Omega)^n$, $\bh_2=\ct_{\nu}(\bu^{in})\in H^{-1/2}(\partial \Omega)^n$. By using  a similar proof of Lemma \ref{lem:unique proof} (or Lemma \ref{lem:unique proof 2}), we can easily obtain the uniqueness of solution to \eqref{eq:scattering1} and derive \eqref{eq:230 bound}.
The proof is complete.
\end{proof}

\section{Proof of Theorem \ref{thm:main1} for Case 1}\label{traction-free obstacle}
In this section, we mainly consider that the traction-free obstacle  $D$ has an $\varepsilon^{1/2}$-realization $(D; \mathcal{C}^0, \rho_0)$ in the sense of Definition~\ref{def:2}, where $\mathcal{C}^0$ is given in the form \eqref{eq:lame2} and $\lambda$, $\mu$, $\rho_0$ satisfy the conditions \eqref{eq:eff2}. Considering an elastic medium $(\Omega; \widetilde{\mathcal{C}}, \widetilde{\rho})$ with $(\widetilde{\mathcal{C}}, \widetilde{\rho})\big|_{\Omega\backslash\overline{D}}=(\mathcal{C}, \rho)\big|_{\Omega\backslash\overline{D}}$ and $(\widetilde{\mathcal{C}}, \widetilde{\rho})\big|_{D}=(\mathcal{C}^0, \rho_0)\big|_{D}$, let $(\mathcal{\widetilde{C}}, \widetilde{\rho})$ be extended into $\mathbb{R}^n\backslash\overline{\Omega}$ such that $(\mathcal{\widetilde{C}}, \widetilde{\rho})=(\mathcal{C}^e, \rho_e)$ in $\mathbb{R}^n\backslash\overline{\Omega}$. Then the medium scattering system described above is given as follows:
\begin{equation}\label{eq:scattering+medium}
\begin{cases}
\mathcal{L}_{\mathcal{\widetilde{C}}}\,\mathbf{\widetilde{u}}+\omega^2\widetilde{\rho}\,\mathbf{\widetilde{u}}={\bff} & \mbox{in}\ \ \mathbb{R}^n,\medskip\\
\mathbf{\widetilde{u}}=\mathbf{u}^{in}+\mathbf{\widetilde{u}}^s & \mbox{in}\ \ \mathbb{R}^n\backslash\overline{\Omega},\medskip\\
\widetilde{\bu}^{-}\big|_{\partial D}=\widetilde{\bu}^{+}\big|_{\partial D},\quad \ct_{\bnu}(\widetilde{\bu}^{-})=\ct_{\bnu}(\widetilde{\bu}^{+})& \mbox{on}\ \ \partial D,\medskip\\
\widetilde{\bu}\big|_{\partial \Omega}=\widetilde{\bu}^{s}\big|_{\partial \Omega}+\bu^{in},\quad \ct_{\bnu}(\widetilde{\bu})=\ct_{\bnu}(\widetilde{\bu}^{s})+\ct_{\bnu}(\bu^{in})&\mbox{on}\ \  \partial{\Omega},\\
\widetilde{\bu}^{\mathrm p,\,s}=-\frac{1}{k_{\mathrm p}^2} \nabla (\nabla \cdot \widetilde{\bu}^{s} ),\quad \widetilde{\bu}^{\mathrm s,\,s}= \frac{1}{k_{\mathrm s} ^2}  \nabla \times  ( \nabla \times   \widetilde{\bu}^{s} ) & \mbox{in}\ \ \mathbb{R}^n\backslash\overline{\Omega}, \medskip\\
{\lim_{|\bx| \to\infty}|\bx|^{(n-1)/2}\big(  \frac{\partial{\widetilde{\bu}^{\mathrm  t,s}}}{\partial{|\bx|}}   -\imath\kappa_{\mathrm  t}\widetilde{\bu}^{\mathrm  t,s}  \big)=\bf 0,} & {\mathrm  t=\mathrm  p,\mathrm s,} \
\end{cases}
\end{equation}
where $\widetilde{\bu}^{-}$ and $\widetilde{\bu}^{+}$ stand for the limits from outside and inside $\partial D$, respectively.
In the following lemma, we first derive the unique solution $\widetilde{\bu}$ of \eqref{eq:scattering+medium}  in regions $B_r\backslash\overline{D}$ and $D$ can be estimated well by $\bu^{in}$ and $\bff$, which plays an important role in the subsequent proof.

\begin{lem}\label{le:prioriestimati}
Let $\widetilde{\bu}$ be the unique solution of \eqref{eq:scattering+medium}. Then there exist positive constants $r_0, \,C_1$ and $C_2$ such that the following estimates hold for all $\varepsilon\ll 1$ and $r\geq r_0$:
\begin{align}
\left\|\widetilde{\bu}\right\|_{H^1(B_r\backslash\overline{D})^n}&\leq C_1\left( \|\bu^{in}\|_{H^1(B_r\backslash\overline{\Omega})^n}  + \|\bff\|_{L^2(B_{r_0}\backslash\overline{\Omega})^n}\right),\label{ineq:Estioutside}\\
\sqrt{\varepsilon}\left\|\widetilde{\bu}\right\|_{H^1(D)^n}&\leq C_2\left( \|\bu^{in}\|_{H^1(B_r\backslash\overline{\Omega})^n}  + \|\bff\|_{L^2(B_{r_0}\backslash\overline{\Omega})^n}\right).\label{ineq:EstiinsideD}
\end{align}
\end{lem}
\begin{proof}

From \eqref{eq:scattering+medium}, we know that $\mathcal{L}_{\mathcal{\widetilde{C}}}\,\mathbf{\widetilde{u}}+\omega^2\widetilde{\rho}\,\mathbf{\widetilde{u}}={\bff}$ in $D$ can be described as
$$
\nabla \cdot(\Ccal^0(\bx) : \nabla \widetilde{\bu})+\omega^2\rho_0 \widetilde{\bu}={\bf 0}.
$$
Multiplying it by $\overline{\widetilde{\bu}}$ and integrating over $D$, and then utilizing the Betti's first formula, we get
\begin{align}\label{eq:inter1}
-\int_D(\Ccal^0:\nabla\overline{\widetilde{\bu}}):\nabla\widetilde{\bu} \,\rmd\bx + \int_{\partial D}\bnu\cdot(\Ccal^0:\nabla\widetilde{\bu})\cdot\overline{\widetilde{\bu} } \,\,\rmd s(\bx)+\omega^2\int_D\rho_0|\widetilde{\bu}|^2\,\rmd\bx=0.
\end{align}
Repeating the similar deduction  for $\mathcal{L}_{\mathcal{\widetilde{C}}}\,\mathbf{\widetilde{u}}+\omega^2\widetilde{\rho}\,\mathbf{\widetilde{u}}={\bff}$ in $\Omega\backslash \overline{D}$, we have
\begin{align}\label{eq:inter2}
\int_{\Omega\backslash\overline{D}}\,(\Ccal(\bx):\nabla\overline{\widetilde{\bu}}):\nabla\widetilde{\bu} \,\rmd\bx =& \int_{\partial \Omega}\bnu\cdot(\Ccal(\bx):\nabla\widetilde{\bu})\cdot\overline{\widetilde{\bu}} \,\rmd s(\bx)
- \int_{\partial D}\bnu\cdot(\Ccal(\bx):\nabla\widetilde{\bu})\cdot\overline{\widetilde{\bu}} \,\rmd s(\bx)\nonumber\\
&+\omega^2\int_{\Omega\backslash\overline{D}}\rho(\bx)|\widetilde{\bu}|^2\,\rmd\bx.
\end{align}
Similarly, we obtain the following integral equation over $B_r\backslash\overline{\Omega}$
\begin{align}\label{eq:inter3}
\int_{B_r\backslash\overline{\Omega}}\,\,(\Ccal^e:\nabla\overline{\widetilde{\bu}^s}):\nabla\widetilde{\bu}^s\,\rmd\bx =& \int_{\partial {B_r}}\bnu\cdot(\Ccal^e:\nabla\widetilde{\bu}^s)\cdot\overline{\widetilde{\bu}^s}\,\, \rmd s(\bx)
- \int_{\partial \Omega}\bnu\cdot(\Ccal^e:\nabla \widetilde{\bu}^s)\cdot\overline{\widetilde{\bu}^s} \,\,\rmd s(\bx)\nonumber\\
&+\omega^2\,\rho_e\int_{B_r\backslash\overline{\Omega}}|\widetilde{\bu}^s|^2\,\rmd\bx-\int_{B_r\backslash\overline{ \Omega}}\bff(x)\cdot \overline{\widetilde{\bu}^s}\,\rmd\bx.
\end{align}
Adding up  the integrals \eqref{eq:inter1}, \eqref{eq:inter2} and \eqref{eq:inter3}, using the transmission conditions given in \eqref{eq:scattering+medium}, we derive that
\begin{align}\label{eq:inter3}
&-\int_D(\Ccal^0:\nabla\overline{\widetilde{\bu}}):\nabla\widetilde{\bu} \, \,\rmd\bx+\int_D \big( \eta_0\,\omega^2\,|\widetilde{\bu}|^2+\imath\tau_0\,\omega^2\,|\widetilde{\bu}|^2\big)\,\rmd\bx-\int_{\Omega\backslash\overline{D}}(\Ccal(x):\nabla\overline{\widetilde{\bu}}):\nabla\widetilde{\bu} \,\rmd\bx\nonumber\\
&+\omega^2\int_{\Omega\backslash\overline{D}}\rho(\bx)|\widetilde{\bu}|^2\,\rmd\bx-\int_{B_r\backslash\overline{\Omega}}\,\,(\Ccal^e:\nabla\overline{\widetilde{\bu}^s}):\nabla\widetilde{\bu}^s \,\rmd\bx+ \int_{\partial B_r}\bnu\cdot(\Ccal^e:\nabla\widetilde{\bu}^s)\cdot\overline{\widetilde{\bu}^s} \,\,\rmd s(\bx)\nonumber\\
&+\int_{\partial {\Omega}}\bnu\cdot(\Ccal^e:\nabla\widetilde{\bu}^s)\cdot\overline{\bu^{in}} \,\,\rmd s(\bx)+\int_{\partial {\Omega}}\bnu\cdot(\Ccal^e:\nabla\bu^{in})\cdot\overline{\widetilde{\bu}^s} \,\rmd s(\bx)+\omega^2\,\int_{B_r\backslash\overline{\Omega}}\,\rho_e|\widetilde{\bu}^s|^2\,\rmd\bx\nonumber\\
&+\int_{\partial {\Omega}}\bnu\cdot(\Ccal^e:\nabla\bu^{in})\cdot\overline{\bu^{in}}\,\rmd s(\bx)=\int_{B_r\backslash\overline{\Omega}}\bff(x)\cdot \overline{\widetilde{\bu}^s}\,\rmd\bx.
\end{align}
Taking the real and imaginary parts of \eqref{eq:inter3}, it is easy to obtain that
\begin{align}
\int_D(\Ccal^0:\nabla\overline{\widetilde{\bu}}):\nabla\widetilde{\bu} \, \rmd\bx&=\int_D \eta_0\,\omega^2|\widetilde{\bu}|^2\rmd\bx-\int_{\Omega\backslash\overline{D}}(\Ccal(\bx):\nabla\overline{\widetilde{\bu}}):\nabla\widetilde{\bu}\, \rmd\bx+\int_{\Omega\backslash\overline{D}}\omega^2\Re\rho|\widetilde{\bu}|^2 \,\rmd\bx\nonumber\\
&\quad-\int_{B_r\backslash\overline{\Omega}}(\Ccal^e:\nabla\overline{\widetilde{\bu}^s}):\nabla\widetilde{\bu}^s\, \rmd\bx+\Re\int_{\partial B_r}\bnu\cdot(\Ccal^e:\nabla\widetilde{\bu}^s)\cdot\overline{\widetilde{\bu}^s}\,\rmd s(\bx)\nonumber\\
&\quad+\Re\int_{\partial {\Omega}}\bnu\cdot(\Ccal^e:\nabla\widetilde{\bu}^s)\cdot\overline{\bu^{in}}\,\rmd s(\bx)
+\Re\int_{\partial {\Omega}}\bnu\cdot(\Ccal^e:\nabla\bu^{in})\cdot\overline{\widetilde{\bu}^s}\,\rmd s(\bx)\nonumber\\
&\quad+\Re\int_{\partial {\Omega}}\bnu\cdot(\Ccal^e:\nabla\bu^{in})\cdot\overline{\bu^{in}}\rmd\bx+\omega^2\,\rho_e\,\int_{B_r\backslash\overline{\Omega}}|\widetilde{\bu}^s|^2\,\rmd\bx\nonumber\\
&\quad-\Re\int_{B_r\backslash\overline{ \Omega}}\bff(\bx)\cdot \overline{\widetilde{\bu}^s}\,\rmd\bx \label{eq:inter4}
\end{align}
and
\begin{align}
\omega^2\tau_0\int_D|\widetilde{\bu}|^2\,\rmd\bx&=-\omega^2\int_{\Omega\backslash\overline{D}}\Im\rho|\widetilde{\bu}|^2\,\rmd\bx -\Im\int_{\partial B_r}\bnu\cdot(\Ccal^e:\nabla\widetilde{\bu}^s)\cdot \overline{\widetilde{\bu}^s}\,\,\rmd s(\bx)\nonumber\\
&\quad -\Im\int_{\partial \Omega}\bnu\cdot(\Ccal^e:\nabla\widetilde{\bu}^s)\cdot \overline{\bu^{in}}\,\rmd s(\bx)-\Im\int_{\partial \Omega}\bnu\cdot(\Ccal^e:\nabla\bu^{in})\cdot \overline{\widetilde{\bu}^s}\,\rmd s(\bx)\nonumber\\
&\quad-\Im\int_{\partial \Omega}\bnu\cdot(\Ccal^e:\nabla\bu^{in})\cdot \overline{\bu^{in}}\,\rmd s(\bx)+\Im\int_{B_r\backslash\overline{\Omega}}\bff(\bx)\cdot\overline{\widetilde{\bu}^s}\,\rmd\bx.\label{eq:inter6}
\end{align}
Using the Kron's inequality (cf. \cite{Marsden1983,Mclean2000}), the definition of the norm of conormal derivatives, H\"older inequality, and Corollary \ref{cor:conormal}, we can directly obtain the following inequalities
\begin{align}
&\varepsilon \left\|\nabla\widetilde{\bu}\right\|^2_{L^2(D)^n}\leq C_1\big(\left\|\widetilde{\bu}\right\|^2_{L^2(D)^n}+\left\|\widetilde{\bu}\right\|^2_{H^1(B_r\backslash\overline{D})^n}+\left\|\bu^{in}\right\|^2_{H^1(B_r\backslash\overline{\Omega})^n}+\left\|\bff\right\|^2_{L^2(B_r\backslash\overline{D})^n}\big), \label{eq:inter5} \\
&\|\widetilde{\bu}\|^2_{L^2(D)^n} \leq C_2\left( \|\widetilde{\bu}\|^2_{H^1(B_r\backslash\overline{D})^n} +\|\bu^{in}\|^2_{H^1(B_r\backslash\overline{\Omega})^n}+ \|\bff\|^2_{L^2(B_r\backslash\overline{\Omega})^n}\right),\label{eq:inter7}
\end{align}
where $C_1$, $C_2$ are positive constants only related to $\lambda_0$, $\mu_0$, $\eta_0$, $\tau_0$, $\omega$, $\Omega$, $B_r$, $\Ccal(\bx)$, $\Ccal^e$ and $\rho$.
Thus, we easily prove the important estimate by adding up \eqref{eq:inter7} and \eqref{eq:inter5},
\begin{align}\label{eq:inter8}
\sqrt{\varepsilon}\|\widetilde{\bu}\|_{H^1(D)^n} \leq \widetilde{C}\left( \|\widetilde{\bu}\|^2_{H^1(B_r\backslash\overline{D})^n}  +\|\bu^{in}\|^2_{H^1(B_r\backslash\overline{\Omega})^n}+ \|\bff\|^2_{L^2(B_r\backslash\overline{\Omega})^n}\right)^{1/2},
\end{align}
where $\widetilde{C}=\max\{\sqrt{C_1},\,\sqrt{C_3}\}$.

In what follows, we shall prove \eqref{ineq:Estioutside} by contradiction. Suppose \eqref{ineq:Estioutside} is not true. Without loss of generality, we assume that for any nonnegative integer $n$, there exists a set of data ($\bff^n$, $\bu_n^{in}$, $\widetilde{\bu}^n$), where $\widetilde{\bu}^n$ is the unique solution of \eqref{eq:scattering+medium} with $\bff^n$ and $\bu_n^{in}$ as inputs, satisfy the restriction
\begin{align}
\left\{ \begin{array}{ll}
\|\bff^n\|_{L^2(B_{r_0}\backslash\overline{\Omega})^n}+\|\bu^{in}_n\|_{H^1(B_r\backslash\overline{\Omega})^n}=1	,\quad \\[5pt]
\widetilde{\bu}^n\to\infty, \quad \mbox{as}\quad \varepsilon\to 0.
\end{array}
 \right.
\end{align}
 We can construct another set of data ($\widetilde{\bff}^n$, $\widetilde{\bg}^{in}$, $\widetilde{\bg}$) as follows:
 \begin{align}
\left\{ \begin{array}{ll}
\widetilde{\bff}^n=\dfrac{\bff^n}{\|\widetilde{\bu}^n\|_{H^1(B_r\backslash\overline{D})^n}},\quad \widetilde{\bg}=\dfrac{\widetilde{\bu}^n}{\|\widetilde{\bu}^n\|_{H^1(B_r\backslash\overline{D})^n}},\\[15pt]
\widetilde{\bg}^{in}=\dfrac{\bu^{in}_n}{\|\widetilde{\bu}^n\|_{H^1(B_r\backslash\overline{D})^n}},\quad \widetilde{\bg}^{s}=\dfrac{\widetilde{\bu}^{n,s}}{\|\widetilde{\bu}^n\|_{H^1(B_r\backslash\overline{D})^n}},\\[15pt]
\widetilde{\bu}^{n,s}= \widetilde{\bu}^n-\bu^{in}_n,\quad \quad\quad\,\,\widetilde{\bg}^{s}= \widetilde{\bg}-\widetilde{\bg}^{in},
\end{array}
 \right.
\end{align}
where $\widetilde{\bg}$ is the unique solution of \eqref{eq:scattering+medium} with $\widetilde{\bff}^n$ and $\widetilde{\bg}^{in}$ as inputs.
Obviously,
\begin{align}\label{eq:A8}
\|\widetilde{\bg}\|_{H^1(B_r\backslash\overline{D})^n}=1,\quad \|\widetilde{\bff}^n\|_{L^2(B_{r_0}\backslash\overline{\Omega})^n}\to 0,\quad \|\widetilde{\bg}^{in}\|_{H^1(B_r\backslash\overline{\Omega})^n}\to 0 \quad\mbox{as}\quad \varepsilon\to 0.
\end{align}
 In fact, we can verify that $(\widetilde{\bg}\big|_{\Omega\backslash\overline{D}},\,\widetilde{\bg}^s\big|_{\mathbb{R}^n\backslash\overline{\Omega}})$ is the unique solution to  problem \eqref{mode:transmission} with $\bp=\bnu\cdot(\Ccal(\bx):\nabla\widetilde{\bg})\Big|_{\partial{D}}$, $\bh_1=\widetilde{\bg}^{in}\big|_{\partial{\Omega}}$, and $\bh_2=\ct_{\bnu}(\widetilde{\bg}^{in})\big|_{\partial{\Omega}}$. According to Lemma \ref{lem:unique proof} and Corollary \ref{cor:conormal}, we have
 \begin{align*}
 \left\{ \begin{array}{ll}
\|\widetilde{\bg}\|_{H^1(B_R\backslash\overline{D})^n}\leq C\big(\big\|\bnu\cdot(\Ccal(x):\nabla\widetilde{\bg})\big\|_{H^{-1/2}(\partial{D})^n}+  \|\widetilde{\bff}^n\|_{L^2(B_{r_0}\backslash\overline{\Omega})^n}  + \|\widetilde{\bg}^{in}\|_{H^{1}(B_r\backslash\overline{\Omega})^n}      \big),\\[10pt]
\Big\|\bnu\cdot(\Ccal(\bx):\nabla\widetilde{\bg})\Big\|_{H^{-1/2}(\partial D)^n}\leq\widetilde{C}\,\varepsilon\,\|\widetilde{\bg}\|_{H^1(D)^n},
 \end{array}
 \right.
  \end{align*}  
  where $C$ and $\widetilde{C}$ are positive constants not relying on $\varepsilon$. Similar to \eqref{eq:inter8}, we can adopt a completely similar argument for  \eqref{eq:scattering+medium} with  the set of data ($\widetilde{\bff}^n$, $\widetilde{\bg}^{in}$, $\widetilde{\bg}$) to derive that

$$\sqrt{\varepsilon}\,\|\widetilde{\bg}\|_{H^1(D)^n}
\leq\,\small{\big(\|\widetilde{\bg}\|_{H^1(B_r\backslash\overline{D})^n} +  \|\widetilde{\bff}^n\|_{L^2(B_{r_0}\backslash\overline{\Omega})^n}  + \|\widetilde{\bg}^{in}\|_{H^{1}(B_r\backslash\overline{\Omega})^n} \big)}. $$
Hence,
$$\big\|\widetilde{\bg}\big\|_{H^1(B_{r}\backslash\overline{D})^n}\to 0\quad\mbox{ as}\quad \varepsilon\to 0^+,$$
which contradicts with the equality $\|\widetilde{\bg}\|_{H^1(B_{r}\backslash\overline{D})^n}=1$. Therefore, the inequality \eqref{ineq:Estioutside} holds.

Substituting  \eqref{ineq:Estioutside} into \eqref{eq:inter8}, using the inequality of arithmetic and geometric means,
one can easily get \eqref{ineq:EstiinsideD}. This completes the proof.

\end{proof}

Next, we shall derive some sharp estimations  of the systems \eqref{eq:scattering1} and \eqref{eq:scattering+medium}, which can indicate that $(\Omega; \widetilde{\mathcal{C}}, \widetilde{\rho})$ is an effective $\varepsilon^{1/2}$-realization of $D\oplus(\Omega\backslash\overline{D}; \mathcal{C}, \rho)$ with traction-free obstacle. Firstly, we would like to show that the conormal derivation of $\Ccal(\bx)$ on the boundary $\partial D$ can be  estimated by the input data.
\begin{prop}\label{th:coderiva_estima}
Let $\widetilde{\bu}\in H^1_{loc}(\mathbb{R}^n)^n$ be the solution to the system \eqref{eq:scattering+medium}. Then there exists a constant $C$ such that the following estimate holds for $\varepsilon\ll 1$ and $r>r_0$:
\begin{align}\label{eq:coderiva_estima}
\big\|  \bnu\cdot[\Ccal(\bx) : \nabla \widetilde{\bu}]\big\|_{H^{-1/2}(\partial D)^n}\leq C\,\varepsilon^{1/2}\left( \|\bu^{in}\|_{H^1(B_r\backslash\overline{\Omega})^n}  + \|\bff\|_{L^2(B_{r_0}\backslash\overline{\Omega})^n}\right).
\end{align}
\end{prop}
\begin{proof}
By using the transmission on $\partial D$ in the system \eqref{eq:scattering+medium}, we have
\begin{align*}
\big\|\bnu\cdot(\Ccal(\bx):\widetilde{\bu})\big\|_{H^{-1/2}(\partial D)^n}&=\big\|\ct_{\bnu}(\widetilde{\bu})\big\|_{H^{-1/2}(\partial D)^n}=\big\|\bnu\cdot(\Ccal^e(\bx):\widetilde{\bu})\big\|_{H^{-1/2}(\partial D)^n}\\
&=\varepsilon\,\big\|\bnu\cdot(\Ccal_0(\bx):\widetilde{\bu})\big\|_{H^{-1/2}(\partial D)^n}\leq C_0\,\varepsilon\,\big\|\widetilde{\bu}\big\|_{H^{1}(D)^n},
\end{align*}
where $\Ccal_0(\bx)$ is a fourth-rank tensor with Lam\'{e} constants both equaling to one.
From Lemma \ref{le:prioriestimati}, it can be verified by straightforward calculations that
\begin{align*}
\big\|\bnu\cdot(\Ccal(\bx):\widetilde{\bu})\big\|_{H^{-1/2}(\partial D)^n}\leq C\varepsilon ^{1/2}\,\Big(   \|\bu^{in}\|_{H^1(B_r\backslash\overline{\Omega})^n}  + \|\bff\|_{L^2(B_{r_0}\backslash\overline{\Omega})^n} \Big).
\end{align*}
The proof is complete.
\end{proof}

The following proposition show that the  solution $\widetilde{\bu}\in H^1_{loc}(\mathbb{R}^n)^n$ to  \eqref{eq:scattering+medium}   can approximate the solution $\bu\in H^1_{loc}(\mathbb{R}^n\backslash\overline{D})^n$ to \eqref{eq:scattering1} with the respect to the parameter $\varepsilon$.

\begin{prop}\label{th:solution_esti}
Suppose $\widetilde{\bu}\in H^1_{loc}(\mathbb{R}^n)^n$ is the solution to system \eqref{eq:scattering+medium} and $\bu\in H^1_{loc}(\mathbb{R}^n\backslash\overline{D})^n$ is the solution to system \eqref{eq:scattering1}. Then there exist two constants $\varepsilon_0$ and $C$ such that the following estimate holds for $\varepsilon\,<\,\varepsilon_0$ and $r>r_0$:
\begin{align}\label{ineq:solution_estiam}
\|\widetilde{\bu}-\bu\|_{H^1(B_r\backslash\overline{D})^n}\leq C\,\varepsilon^{1/2}\left( \|\bu^{in}\|_{H^1(B_r\backslash\overline{\Omega})^n}  + \|\bff\|_{L^2(B_{r_0}\backslash\overline{\Omega})^n}\right).
\end{align}
\end{prop}
\begin{proof}
Let $\bv=\widetilde{\bu}-\bu$, where $\widetilde{\bu}$ and $\bu$ are the total fields of system \eqref{eq:scattering+medium} and system \eqref{eq:scattering1}, respectively. We can easily verify that $(\bv\big|_{\Omega\backslash\overline{D}},\,\bv^s\big|_{\mathbb{R}^n\backslash\overline{\Omega}})$ is the unique solution of system \eqref{mode:transmission} with the boundary conditions: $\bff=\bh_1=\bh_2=0$, $\bp=\bnu\cdot(\Ccal(\bx):\nabla\bv)=\bnu\cdot(\Ccal(\bx):\nabla\widetilde{\bu})$. Combining $\bv^s=\widetilde{\bu}^s-\bu^s=(\widetilde{\bu}-\bu^{in})-(\bu-\bu^{in})=\bv$ and Lemma \ref{lem:unique proof} with Lemma \ref{th:coderiva_estima}, we obtain
\begin{align*}
\|\bv\|_{H^1(B_r\backslash\overline{D})^n}&\leq C \big\|\bnu\cdot(\Ccal(\bx):\nabla\bv)\big\|_{H^{-1/2}(\partial D)^n}
=C\big\|\bnu\cdot(\Ccal(\bx):\nabla\widetilde{\bu})\big\|_{H^{-1/2}(\partial D)^n}\nonumber\\[5pt]
&\leq C\,\varepsilon^{1/2}\,\bigg(\|\bu^{in}\|_{\bH^1(B_r\backslash\overline{\Omega})}  + \|\bff\|_{L^2(B_{r_0}\backslash\overline{\Omega})^n}\bigg).
\end{align*}
The completes the proof.
\end{proof}

We are in the position to give the proof of Theorem \ref{thm:main1}  for Case 1.
\begin{proof}[Proof of Theorem \ref{thm:main1} for  Case 1]
 Let $\bv=\widetilde{\bu}-\bu$ and $\bv^s=\widetilde{\bu}^s-\bu^s$. We note that $\bv^s=\bv$. We use the following explicit expressions of the scattering amplitude of $\widetilde{\bu}^s$ and $\bu^s$ given by \cite{Alves2002,Arens2001} (see more details in  \cite{Alves1997,Dassios1984,Dassios1987,Dassios1995}):
\begin{eqnarray}
\widetilde{\mathbf{u}}^{\mathrm p,\infty} &=&\dfrac{\kappa_{\mathrm p}^2}{4\pi\omega^2}\int_{\partial B_r}\bigg\{\Big(\ct^{\by}_{\bnu}\big(\hat{\bx}\hat{\bx}^\top e^{-\imath\kappa_{\mathrm p}\hat{\bx}\cdot\by}\big)\Big)^\top \cdot\widetilde{\bu}^s(\by)\nonumber\\
&&-\big(\hat{\bx}\,\hat{\bx}^{\top}e^{-\imath\kappa_{\mathrm p}\hat{\bx}\cdot\by}\big)\cdot\ct_{\bnu}\widetilde{\bu}^s(\by)   \bigg\} \rmd s(y), \quad \, \hat{\bx}\in \mathbb{S}^{n-1},\label{eq:amplitude P_part-varepsilon-medium} \\
\mathbf{u}^{\mathrm p,\infty} &=&\dfrac{\kappa_{\mathrm p}^2}{4\pi\omega^2}\int_{\partial B_r}\bigg\{\Big(\ct^{\by}_{\bnu}\big(\hat{\bx}\hat{\bx}^\top e^{-\imath\kappa_{\mathrm p}\hat{\bx}\cdot\by}\big)\Big)^\top\cdot \bu^s(\by)\nonumber\\
&&-\big(\hat{\bx}\,\hat{\bx}^{\top}e^{-\imath\kappa_{\mathrm p}\hat{\bx}\cdot\by}\big)\cdot\ct_{\bnu}\bu^s(\by)  \bigg\} \rmd s(y), \quad \, \hat{\bx}\in \mathbb{S}^{n-1},
\label{eq:amplitude P_part} \\
\widetilde{\mathbf{u}}^{\mathrm s,\infty} &=&\dfrac{\kappa_{\mathrm s}^2}{4\pi\omega^2}\int_{\partial B_r}\bigg\{\Big\{\ct^{\by}_{\bnu}\Big(\big(I-\hat{\bx}\hat{\bx}^\top\big) e^{-\imath\kappa_{\mathrm s}\hat{\bx}\cdot\by}\Big)\Big\}^\top \cdot\widetilde{\bu}^s(\by) \nonumber\\
&&-\Big(I- \hat{\bx}\,\hat{\bx}^{\top}   \Big)e^{-\imath\kappa_{\mathrm s}\hat{\bx}\cdot\by} \cdot\ct_{\bnu}\widetilde{\bu}^s(\by)\bigg\} \rmd s(y), \quad \, \hat{\bx}\in \mathbb{S}^{n-1},\label{eq:amplitude S_part-varepsilon-medium} \\
\mathbf{u}^{\mathrm s,\infty} &=&\dfrac{\kappa_{\mathrm s}^2}{4\pi\omega^2}\int_{\partial B_r}\bigg\{ \Big\{\ct^{\by}_{\bnu}\Big(\big(I-\hat{\bx}\hat{\bx}^\top\big) e^{-\imath\kappa_{\mathrm s}\hat{\bx}\cdot\by}\Big)\Big\}^\top \cdot\bu^s(\by)\nonumber\\
&&-\Big(I- \hat{\bx}\,\hat{\bx}^{\top}   \Big)e^{-\imath\kappa_{\mathrm s}\hat{\bx}\cdot\by} \cdot\ct_{\bnu}\bu^s(\by)    \bigg\} \rmd s(y), \quad \, \hat{\bx}\in \mathbb{S}^{n-1},\label{eq:amplitude S_part-varepsilon}
\end{eqnarray}
where $\widetilde{\mathbf{u}}^{\mathrm p,\infty}$ and $\mathbf{u}^{\mathrm p,\infty}$ respectively are the longitudinal far-field patterns of $\widetilde{\bu}^s$ and $\bu^s$, which are normal to $\mathbb{S}^{n-1}$, whereas $\widetilde{\mathbf{u}}^{\mathrm s,\infty}$ and $\mathbf{u}^{\mathrm s,\infty}$ respectively are the transversal far-field patterns of $\widetilde{\bu}^s$ and $\bu^s$, which are tangential to $\mathbb{S}^{n-1}$. 
Then
\begin{align}
\widetilde{\mathbf{u}}^{\mathrm p,\infty}-\mathbf{u}^{\mathrm p,\infty}&=
\dfrac{\kappa_p^2}{4\pi\omega^2}\int_{\partial B_r}\bigg\{\Big(\ct^{\by}_{\bnu}\big(\hat{\bx}\hat{\bx}^\top e^{-\imath\kappa_{\mathrm p}\hat{\bx}\cdot\by}\big)\Big)^\top \cdot\bv^s(\by)\nonumber\\
&\quad- \big(\hat{\bx}\,\hat{\bx}^{\top}e^{-\imath\kappa_{\mathrm p}\hat{\bx}\cdot\by}\big)\cdot\ct_{\bnu}\bv^s(\by)\bigg\}\rmd s(y),\,\, \hat{\bx}\in \mathbb{S}^{n-1},\label{eq:amplitudeVP}\\
\widetilde{\mathbf{u}}^{\mathrm s,\infty}-\mathbf{u}^{\mathrm s,\infty}&=
\dfrac{\kappa_{\mathrm s}^2}{4\pi\omega^2}\int_{\partial B_r}\bigg\{\Big\{\ct^{\by}_{\bnu}\Big(\big(I-\hat{\bx}\hat{\bx}^\top\big) e^{-\imath\kappa_{\mathrm s}\hat{\bx}\cdot\by}\Big)\Big\}^\top \cdot\bv^s(\by)\nonumber\\
&\quad-\Big(I- \hat{\bx}\,\hat{\bx}^{\top}   \Big)e^{-\imath\kappa_{\mathrm s}\hat{\bx}\cdot\by} \cdot\ct_{\bnu}\bv^s(\by)\bigg\}\rmd s(y),\,\, \hat{\bx}\in \mathbb{S}^{n-1}.\label{eq:amplitudeVS}
\end{align}
Let $A=\hat{\bx}\hat{\bx}^\top e^{-\imath\kappa_{\mathrm p}\hat{\bx}\cdot\by}=[A_1,\,A_2,\,A_3]$, $B=\big(I-\hat{\bx}\hat{\bx}^\top\big) e^{-\imath\kappa_{\mathrm s}\hat{\bx}\cdot\by}=[B_1,\,B_2,\,B_3]$, where $A_j$ is the $j$th column of $A$ and $B_j$ is the $j$th column of $B$, $j=1,2,3$. Using the following fact that, for any vector field $\bpsi$,
\begin{align*}
 \left\{ \begin{array}{ll}
 \big|\bnu\cdot\nabla\bpsi\big|\leq C_1\big|    \nabla\bpsi\big|,\\[5pt]
 \big|\bnu\,\,\nabla\cdot\bpsi\big| \leq C_2\big|    \nabla\bpsi\big|,\\[5pt]
 \big|\bnu\times \nabla\times \bpsi\big|\leq C_3\big|    \nabla\bpsi\big|,
  \end{array}
 \right.
\end{align*}
where $|\cdot|$ denote Frobenius norm for a matrix or Euclidean norm for a vector, $C_1,\,C_2,\,C_3$ are positive constants, $\bnu$ denotes the unit outward normal to the boundary, we have
$$
\big|\ct^\by_{\bnu}A_j\big|\leq C_4,\quad \big|\nabla A_j\big|\leq C_5\,\kappa_{\mathrm p},\quad \big|A_j\big|\leq C_6\, r
$$
for $j=1,2,3$, where $C_4,\,C_5,\,C_6$ are positive constants. One can derive the following estimate by using trace theorem, Theorem \ref{th: thm1} and Proposition \ref{th:solution_esti},
\begin{align*}
\Big\|\widetilde{\mathbf{u}}^{\mathrm p,\infty} -\mathbf{u}^{\mathrm p,\infty}\big\|_{C(\mathbb{S}^{n-1})^n}&\leq C_7\Big( \|\bv\|_{H^{1/2}(B_r\backslash\overline{D})^n}  +  \|\ct^{\by}_{\bnu}\bv\|_{H^{-1/2}(\partial B_r)^n} \Big)\\[5pt]
&\leq C_8\Big( \|\bv\|_{H^{1}(B_r\backslash\overline{D})^n}  +  \|\ct^{\by}_{\bnu}\bv\|_{H^{-1/2}(\partial B_r)^n}        \Big)\\[5pt]
&\leq C_9\Big( \|\bv\|_{H^{1}(B_r\backslash\overline{D})^n} +C_{10}  \|\bv\|_{H^{1}(B_r\backslash\overline{D})^n}    \Big)\\[5pt]
&\leq C_{10}\,\varepsilon^{1/2}\,\bigg(\|\bu^{in}\|_{H^1(B_r\backslash\overline{\Omega})^n}  + \|\bff\|_{L^2(B_{r_0}\backslash\overline{\Omega})^n}\bigg),
\end{align*}
where $C_7,\,C_8,\,C_9, \,C_{10}$ are positive constants  depening only on $\omega$, $\kappa_p$, $\Ccal(\bx)$, $\Ccal^e$, $B_{r_0}\backslash\overline{\Omega}$ and $B_r\backslash\overline{D}$.
Similarly, we obtain that
\begin{align*}
\Big\|\widetilde{\mathbf{u}}^{\mathrm p,\infty}-\mathbf{u}^{\mathrm p,\infty}\big\|_{C(\mathbb{S}^{n-1})^n}\leq C_{11} \, \varepsilon^{1/2}\,\bigg(\|\bu^{in}\|_{H^1(B_r\backslash\overline{\Omega})^n}  + \|\bff\|_{L^2(B_{r_0}\backslash\overline{\Omega})^n}\bigg),
\end{align*}
where $C_{11}=C_{11}(\omega,\,\kappa_s,\,B_{r_0}\backslash\overline{\Omega},\,B_r\backslash\overline{D},\, \Ccal(\bx),\,\Ccal^e)$. Hence,
\begin{align*}
\Big\|\widetilde{\mathbf{u}}^{\infty} -\mathbf{u}^{\infty}\big\|_{C(\mathbb{S}^2)^n}\leq &\Big\|\widetilde{\mathbf{u}}^{\mathrm p,\infty} -\mathbf{u}^{\mathrm p,\infty}\big\|_{C(\mathbb{S}^{n-1})^n}+\Big\|\widetilde{\mathbf{u}}^{\mathrm s,\infty} -\mathbf{u}^{\mathrm s,\infty}\big\|_{C(\mathbb{S}^{n-1})^n}\\
\leq&C\, \varepsilon^{1/2}\,\bigg(\|\bu^{in}\|_{H^1(B_r\backslash\overline{\Omega})^n}  + \|\bff\|_{L^2(B_{r_0}\backslash\overline{\Omega})^n}\bigg),
\end{align*}
where $C=C(\omega,\,\kappa_p,\,\kappa_s,\,B_r\backslash\overline{D},\,B_{r_0}\backslash\overline{\Omega},\, \Ccal(\bx),\,\Ccal^e)\in \mathbb R_+ $. The proof is complete.
\end{proof}

\section{Proof of Theorem \ref{thm:main1} for  Case 2}\label{rigid obstacle}

In this section, we are committed to proving that $(D; \mathcal{C}^0, \rho_0)$ is an $\varepsilon^{1/2}$-realization of the rigid obstacle $D$ in the sense of Definition~\ref{def:2}, where $\mathcal{C}^0$ is given in the form \eqref{eq:lame2} and $\lambda$, $\mu$, $\rho_0$ satisfy the conditions \eqref{eq:eff2_2}.  An elastic medium $(\Omega ; \widetilde{\mathcal{C}}, \widetilde{\rho})$ is considered, which satisfies that $(\widetilde{\mathcal{C}}, \widetilde{\rho})\big|_{\Omega\backslash\overline{D}}=(\mathcal{C}, \rho)\big|_{\Omega\backslash\overline{D}}$, $(\widetilde{\mathcal{C}}, \widetilde{\rho})\big|_{D}=(\mathcal{C}^0, \rho_0)\big|_{D}$ and $(\mathcal{\widetilde{C}}, \widetilde{\rho})\big|_{\mathbb{R}^n\backslash\overline{\Omega}}=(\mathcal{C}^e, \rho_e)\big|_{\mathbb{R}^n\backslash\overline{\Omega}}$. Consider the medium scattering system \eqref{eq:scattering+medium} except that $( \mathcal{C}^0, \rho_0)$ satisfying \eqref{eq:eff2} is replaced by $( \mathcal{C}^0, \rho_0)$ with the parameters in \eqref{eq:eff2_2}.

In the following lemma, we derive that the unique solution $\widetilde{\bu}$ of \eqref{eq:scattering+medium}  in regions $B_r\backslash\overline{D}$ and $D$ can be estimated well by $\bu^{in}$ and $\bff$, which play an important role in the subsequent  proof.
\begin{lem}\label{le:prioriestimati 2}
Let $\widetilde{\bu}$ be the unique solution of \eqref{eq:scattering+medium} with $( \mathcal{C}^0, \rho_0)$ satisfying \eqref{eq:eff2_2}. Then there exist positive constants $r_0, \,C_1, \,C_2$ such that the following estimate holds for all $\varepsilon\ll 1$ and $r\geq r_0$:
\begin{align}
\left\|\widetilde{\bu}\right\|_{H^1(B_r\backslash\overline{D})^n}&\leq C_1\left( \|\bu^{in}\|_{H^1(B_r\backslash\overline{\Omega})^n}  + \|\bff\|_{L^2(B_{r_0}\backslash\overline{\Omega})^n}\right),\label{ineq:Estioutside 2}\\
\left\|\widetilde{\bu}\right\|_{H^1(D)^n}&\leq C_2\,\varepsilon^{1/2}\left( \|\bu^{in}\|_{H^1(B_r\backslash\overline{\Omega})^n}  + \|\bff\|_{L^2(B_{r_0}\backslash\overline{\Omega})^n}\right).\label{ineq:EstiinsideD 2}
\end{align}
\end{lem}

\begin{proof}
Multiplying $\mathcal{L}_{\mathcal{\widetilde{C}}}\,\mathbf{\widetilde{u}}+\omega^2\widetilde{\rho}\,\mathbf{\widetilde{u}}={\bff}$  by $\overline{\widetilde{\bu}}$ and integrating it over $D$, $\Omega\backslash \overline{D}$, $B_r\backslash\overline{\Omega}$, respectively. By adding up them and the transmissions on $\partial D$ and $\partial \Omega$, we have

\begin{align}\label{eq:inter3 2}
&-\int_D(\Ccal^0:\nabla\overline{\widetilde{\bu}}):\nabla\widetilde{\bu} \, \,\rmd\bx+\int_D \big( \eta_0\,\omega^2\,|\widetilde{\bu}|^2+\imath\varepsilon^{-1}\tau_0\,\omega^2\,|\widetilde{\bu}|^2\big)\,\rmd\bx-\int_{\Omega\backslash\overline{D}}(\Ccal(x):\nabla\overline{\widetilde{\bu}}):\nabla\widetilde{\bu} \,\rmd\bx\nonumber\\
&+\omega^2\int_{\Omega\backslash\overline{D}}\rho(\bx)|\widetilde{\bu}|^2\,\rmd\bx-\int_{B_r\backslash\overline{\Omega}}\,\,(\Ccal^e:\nabla\overline{\widetilde{\bu}^s}):\nabla\widetilde{\bu}^s \,\rmd\bx+ \int_{\partial B_r}\bnu\cdot(\Ccal^e:\nabla\widetilde{\bu}^s)\cdot\overline{\widetilde{\bu}^s} \,\,\rmd s(\bx)\nonumber\\
&+\int_{\partial {\Omega}}\bnu\cdot(\Ccal^e:\nabla\widetilde{\bu}^s)\cdot\overline{\bu^{in}} \,\,\rmd s(\bx)+\int_{\partial {\Omega}}\bnu\cdot(\Ccal^e:\nabla\bu^{in})\cdot\overline{\widetilde{\bu}^s} \,\rmd s(\bx)+\omega^2\,\int_{B_r\backslash\overline{\Omega}}\,\rho_e|\widetilde{\bu}^s|^2\,\rmd\bx\nonumber\\
&+\int_{\partial {\Omega}}\bnu\cdot(\Ccal^e:\nabla\bu^{in})\cdot\overline{\bu^{in}}\,\rmd s(\bx)=\int_{B_r\backslash\overline{\Omega}}\bff(x)\cdot \overline{\widetilde{\bu}^s}\,\rmd\bx.
\end{align}
Taking the real and imaginary parts of  \eqref{eq:inter3 2}, it is easy to obtain that
\begin{align}
\int_D(\Ccal^0:\nabla\overline{\widetilde{\bu}}):\nabla\widetilde{\bu} \, \rmd\bx&=\int_D \,\eta_0\,\omega^2|\widetilde{\bu}|^2\rmd\bx-\int_{\Omega\backslash\overline{D}}(\Ccal(\bx):\nabla\overline{\widetilde{\bu}}):\nabla\widetilde{\bu}\, \rmd\bx+\int_{\Omega\backslash\overline{D}}\omega^2\Re\rho|\widetilde{\bu}|^2 \,\rmd\bx\nonumber\\
&\quad-\int_{B_r\backslash\overline{\Omega}}(\Ccal^e:\nabla\overline{\widetilde{\bu}^s}):\nabla\widetilde{\bu}^s\, \rmd\bx+\Re\int_{\partial B_r}\bnu\cdot(\Ccal^e:\nabla\widetilde{\bu}^s)\cdot\overline{\widetilde{\bu}^s}\,\rmd s(\bx)\nonumber\\
&\quad+\Re\int_{\partial {\Omega}}\bnu\cdot[\Ccal^e:\nabla\widetilde{\bu}^s]\cdot\overline{\bu^{in}}\,\rmd s(\bx)
+\Re\int_{\partial {\Omega}}\bnu\cdot(\Ccal^e:\nabla\bu^{in})\cdot\overline{\widetilde{\bu}^s}\,\rmd s(\bx)\nonumber\\
&\quad+\Re\int_{\partial {\Omega}}\bnu\cdot(\Ccal^e:\nabla\bu^{in})\cdot\overline{\bu^{in}}\rmd\bx+\omega^2\,\rho_e\,\int_{B_r\backslash\overline{\Omega}}|\widetilde{\bu}^s|^2\,\rmd\bx\nonumber\\
&\quad-\Re\int_{B_r\backslash\overline{ \Omega}}\bff(\bx)\cdot \overline{\widetilde{\bu}^s}\,\rmd\bx\label{eq:inter4}
\end{align}
and
\begin{align}
\omega^2\varepsilon^{-1}\,\tau_0\int_D|\widetilde{\bu}|^2\,\rmd\bx&=-\omega^2\int_{\Omega\backslash\overline{D}}\Im\rho|\widetilde{\bu}|^2\,\rmd\bx -\Im\int_{\partial B_r}\bnu\cdot(\Ccal^e:\nabla\widetilde{\bu}^s)\cdot \overline{\widetilde{\bu}^s}\,\,\rmd s(\bx)\nonumber\\
&\quad -\Im\int_{\partial \Omega}\bnu\cdot[\Ccal^e:\nabla\widetilde{\bu}^s]\cdot \overline{\bu^{in}}\,\rmd s(\bx)-\Im\int_{\partial \Omega}\bnu\cdot(\Ccal^e:\nabla\bu^{in})\cdot \overline{\widetilde{\bu}^s}\,\rmd s(\bx)\nonumber\\
&\quad-\Im\int_{\partial \Omega}\bnu\cdot(\Ccal^e:\nabla\bu^{in})\cdot \overline{\bu^{in}}\,\rmd s(\bx)+\Im\int_{B_r\backslash\overline{\Omega}}\bff(\bx)\cdot\overline{\widetilde{\bu}^s}\,\rmd\bx.\label{eq:inter6 2}
\end{align}
Since $\mathcal{C}^0$ satisfies the uniform Legendre ellipticity condition \eqref{eq:ellip1} and $\lambda=\varepsilon^{-1}\lambda_0$, $\mu=\varepsilon^{-1}\mu_0$,
$$
\int_D(\Ccal^0:\nabla\overline{\widetilde{\bu}}):\nabla\widetilde{\bu} \,\, dx\geq C_0\, \varepsilon^{-1}\|\nabla\widetilde{\bu} \|^2_{L^2(D)^n},
$$
where $C_0$ only depends on $\lambda_0$ and $\mu_0$. And then we can directly obtain
\begin{align}
 \left\|\nabla\widetilde{\bu}\right\|^2_{L^2(D)^n}&\leq C_1\,\varepsilon\,\big(\left\|\widetilde{\bu}\right\|^2_{H^1(B_r\backslash\overline{D})^n}+\left\|\bu^{in}\right\|^2_{H^1(B_r\backslash\overline{\Omega})^n}+\left\|\bff\right\|^2_{L^2(B_r\backslash\overline{D})^n}\big),\label{eq:inter5 2}\\
\|\widetilde{\bu}\|^2_{L^2(D)^n}& \leq C_2\,\varepsilon\,\left( \|\widetilde{\bu}\|^2_{H^1(B_r\backslash\overline{D})^n} +\|\bu^{in}\|^2_{H^1(B_r\backslash\overline{\Omega})^n}+ \|\bff\|^2_{L^2(B_r\backslash\overline{\Omega})^n}\right),\label{eq:inter7 2}
\end{align}
where $C_1$, $C_2$ are positive constants not related to $\varepsilon$.
Thus,
\begin{align}\label{eq:inter8 2}
\|\widetilde{\bu}\|_{H^1(D)^n} \leq C \, \varepsilon^{1/2}\,\left( \|\widetilde{\bu}\|^2_{H^1(B_r\backslash\overline{D})^n}  +\|\bu^{in}\|^2_{H^1(B_r\backslash\overline{\Omega})^n}+ \|\bff\|^2_{L^2(B_r\backslash\overline{\Omega})^n}\right)^{1/2}.
\end{align}
As we did in proving \eqref{ineq:Estioutside}, we can construct two sets of data ($\bff^n$, $\bu_n^{in}$, $\widetilde{\bu}^n$) and ($\widetilde{\bff}^n$, $\widetilde{\bg}^{in}$, $\widetilde{\bg}$) as follows, where $\widetilde{\bu}^n$ is the unique solution of \eqref{eq:scattering+medium} with $\bff^n$ and $\bu_n^{in}$ as inputs and $\widetilde{\bg}$ is the unique solution of \eqref{eq:scattering+medium} with $\widetilde{\bff}^n$ and $\widetilde{\bg}^{in}$ as inputs,
\begin{align}\label{data:1 2}
\left\{ \begin{array}{ll}
\widetilde{\bff}^n=\dfrac{\bff^n}{\|\widetilde{\bu}^n\|_{H^1(B_r\backslash\overline{D})^n}},\quad \widetilde{\bg}=\dfrac{\widetilde{\bu}^n}{\|\widetilde{\bu}^n\|_{H^1(B_r\backslash\overline{D})^n}},\\[15pt]
\widetilde{\bg}^{in}=\dfrac{\bu^{in}_n}{\|\widetilde{\bu}^n\|_{H^1(B_r\backslash\overline{D})^n}},\quad \widetilde{\bg}^{s}=\dfrac{\widetilde{\bu}^{n,s}}{\|\widetilde{\bu}^n\|_{H^1(B_r\backslash\overline{D})^n}},\\[15pt]
\widetilde{\bu}^{n,s}= \widetilde{\bu}^n-\bu^{in}_n,\quad \quad\quad\,\,\widetilde{\bg}^{s}= \widetilde{\bg}-\widetilde{\bg}^{in},\\[5pt]
\|\bff^n\|_{L^2(B_{r_0}\backslash\overline{\Omega})^n}+\|\bu^{in}_n\|_{H^1(B_r\backslash\overline{\Omega})^n}=1	,\quad \\[5pt]
\|\widetilde{\bu}^n\|_{H^1(B_r\backslash\overline{D})^n}\to\infty, \quad \mbox{as}\quad \varepsilon\to 0.
\end{array}
 \right.
\end{align}
And then we have
\begin{align}
&\|\widetilde{\bg}\|_{H^1(B_r\backslash\overline{D})^n}=1,\quad \|\widetilde{\bff}^n\|_{L^2(B_{r_0}\backslash\overline{\Omega})^n}\to 0,\quad \|\widetilde{\bg}^{in}\|_{H^1(B_r\backslash\overline{\Omega})^n}\to 0 \quad\mbox{as}\quad \varepsilon\to 0^{+},\label{eq:A8 2}\\
&\|\widetilde{\bg}\|_{H^1(D)^n}
\leq C\,\varepsilon^{1/2}\,\small{\big(\|\widetilde{\bg}\|_{H^1(B_r\backslash\overline{D})^n} +  \|\widetilde{\bff}^n\|_{L^2(B_{r_0}\backslash\overline{\Omega})^n}  + \|\widetilde{\bg}^{in}\|_{H^{1}(B_r\backslash\overline{\Omega})^n} \big)}.\label{eq:A9 2}
\end{align}
From Lemma \ref{lem:unique proof 2}, $(\widetilde{\bg}\big|_{\Omega\backslash\overline{D}},\,\widetilde{\bg}^s\big|_{\mathbb{R}^n\backslash\overline{\Omega}})$ is the unique solution of \eqref{mode:transmission} with $\bp=\widetilde{\bg}\Big|_{\partial{D}}$, $\bh_1=\widetilde{\bg}^{in}\big|_{\partial{\Omega}}$, and $\bh_2=\ct_{\bnu}(\widetilde{\bg}^{in})\big|_{\partial{\Omega}}$ such that
 \begin{align*}
\|\widetilde{\bg}\|_{H^1(B_r\backslash\overline{D})^n}&\leq d\big(\big\|\widetilde{\bg}\|_{H^{1/2}(\partial{D})^n}+  \|\widetilde{\bff}^n\|_{L^2(B_{r_0}\backslash\overline{\Omega})^n}  + \|\widetilde{\bg}^{in}\|_{H^{1}(B_r\backslash\overline{\Omega})^n}      \big),\\
&\leq \widetilde{d}\big(\big\|\widetilde{\bg}\|_{H^{1}(D)^n}+  \|\widetilde{\bff}^n\|_{L^2(B_{r_0}\backslash\overline{\Omega})^n}  + \|\widetilde{\bg}^{in}\|_{H^{1}(B_r\backslash\overline{\Omega})^n}      \big),
 \end{align*}
where $d$ and $\widetilde{d}$ are positive constants not relying on $\varepsilon$. Hence, we can see that
$$\big\|\widetilde{\bg}\big\|_{H^1(B_{r}\backslash\overline{D})^n}\to 0\quad\mbox{ as}\quad \varepsilon\to 0,$$
which contradicts with the equality $\|\widetilde{\bg}\|_{H^1(B_{r}\backslash\overline{D})^n}=1$. Hence, the inequality \eqref{ineq:Estioutside 2} holds.

Next, we prove \eqref{ineq:EstiinsideD 2}. From \eqref{ineq:Estioutside 2} and  \eqref{eq:inter8 2}, it is easy to obtain that \eqref{ineq:EstiinsideD 2} holds. This completes the proof.


\end{proof}

\begin{prop}\label{th:solution_esti 2}
Suppose $\widetilde{\bu}\in H^1_{loc}(\mathbb{R}^n)^n$ is the solution to system \eqref{eq:scattering+medium} and $\bu\in H^1_{loc}(\mathbb{R}^n\backslash\overline{D})^n$ is the solution to system \eqref{eq:scattering1}. Then there exist a  constant $C$ such that the following estimate holds for $\varepsilon\,\ll 1$ and $r>r_0$:
\begin{align}\label{ineq:solution_estiam 2}
\|\widetilde{\bu}-\bu\|_{H^1(B_r\backslash\overline{D})^n}\leq C\,\varepsilon^{1/2}\left( \|\bu^{in}\|_{H^1(B_r\backslash\overline{\Omega})^n}  + \|\bff\|_{L^2(B_{r_0}\backslash\overline{\Omega})^n}\right).
\end{align}
\end{prop}

\begin{proof}
Let $\bv=\widetilde{\bu}-\bu$, where $\widetilde{\bu}$ and $\bu$ are the total fields of system \eqref{eq:scattering+medium} and system \eqref{eq:scattering1}, respectively. We can easily verify that $(\bv\big|_{\Omega\backslash\overline{D}},\,\bv^s\big|_{\mathbb{R}^n\backslash\overline{\Omega}})$ is the unique solution of system \eqref{mode:transmission} with the boundary conditions: $\bff=\bh_1=\bh_2=0$, $\bp=\bv\big|_{\partial D}=\widetilde{\bu}\big|_{\partial D}$. By using Lemma \ref{lem:unique proof 2}, Trace Theorem and Lemma \ref{le:prioriestimati 2}, we obtain
\begin{align*}
\|\bv\|_{H^1(B_r\backslash\overline{D})^n}&\leq 
C_1\big\|\widetilde{\bu}\big\|_{H^{1/2}(\partial D)^n}\leq C_2\|\widetilde{\bu}\|_{H^{1}(D)^n}
\leq C\,\varepsilon^{1/2}\,\big(\|\bu^{in}\|_{\bH^1(B_r\backslash\overline{\Omega})}  + \|\bff\|_{L^2(B_{r_0}\backslash\overline{\Omega})^n}\big).
\end{align*}
The proof is complete.
\end{proof}

Using Lemma \ref{lem:unique proof 2} and Proposition \ref{th:solution_esti 2} we can establish Theorem \ref{thm:main1} for Case 2 by using similar  arguments of Theorem \ref{thm:main1} for Case 1. The detailed proof is skipped.

\section*{Acknowledgement}

The research of Z Bai is partially supported by the National Natural Science Foundation of China  under grant  11671337 and the Fundamental Research Funds for the Central Universities under grant  20720180008. The work  of H Diao was supported in part by the National Natural Science Foundation of China under grant 11001045 and the Fundamental Research Funds for the Central Universities under the grant 2412017FZ007.  The work of H Liu was supported by the startup fund from City University of Hong Kong and the Hong Kong RGC General Research Fund (projects 12301420, 12302919 and 12301218).

\end{document}